\documentclass[11pt]{article}
\usepackage[pagebackref,colorlinks=true,pdfpagemode=none,urlcolor=blue,linkcolor=blue,citecolor=blue]{hyperref}

\usepackage{amsthm, amssymb, bm, bbm}
\usepackage{soul, color,cancel}
\usepackage[]{amsmath}
\usepackage[]{amsfonts}
\usepackage[]{fancyhdr}
\usepackage[]{graphicx}
\graphicspath{{/EPSF/}{../figures/}{./figures/}}

\newtheorem{theorem}{Theorem}[section]
\newtheorem{lemma}[theorem]{Lemma}
\newtheorem{proposition}[theorem]{Proposition}
\newtheorem{corollary}[theorem]{Corollary}

\newtheorem{remark}[theorem]{Remark}

\def \Cm {\mathbb{C}}

\def \Nm {\mathbb{N}}
\def \Rm {\mathbb{R}}
\def \Sm {\mathbb{S}}

\def \Zm {\mathbb{Z}}

\newcommand      {\bone}        {{ \mathbbm 1}}

\def\C{\mathcal{C}}
\def\D{\mathcal{D}}

\def\V{\mathcal{V}}

\def\tu{\tilde{u}}

\newcommand{\where}{\quad\text{ where }}
\newcommand{\qandq}{\quad\text{ and }\quad}

\newcommand{\cout}[1]{}

\newcommand{\x}{\mathrm{x}}
\newcommand{\y}{\mathrm{y}}

\newcommand{\sgn}[1]{\,{\rm sign}(#1)}

\newcommand{\mat}[4]{\left[ \begin{array}{cc} #1 & #2 \\ #3 & #4 \end{array} \right]}

\renewcommand{\Im}{\mathbb{I}\mathrm{m} }

\title{Inversion of the attenuated geodesic X-ray transform over functions and vector fields on simple surfaces}
\author{Fran\c cois Monard\thanks{Department of Mathematics, University of Michigan, Ann Arbor MI, 48109; monard@umich.edu}}

\begin{document}
\maketitle

\begin{abstract}
    We derive explicit reconstruction formulas for the attenuated geodesic X-ray transform over functions and, in the case of non-vanishing attenuation, vector fields, on a class of simple Riemannian surfaces with boundary. These formulas partly rely on new explicit approaches to construct continuous right-inverses for backprojection operators (and, in turn, holomorphic integrating factors), which were previously unavailable in a systematic form. The reconstruction of functions is presented in two ways, the latter one being motivated by numerical considerations and successfully implemented at the end. Constructing the right-inverses mentioned require that certain Fredholm equations, first appearing in \cite{Pestov2004}, be invertible. Whether this last condition reduces the applicability of the overall approach to a strict subset of simple surfaces remains open at present.
\end{abstract}

\section{Introduction}
We consider explicit inversion formulas for the two-dimensional attenuated X-ray (or, equivalently in two dimensions, Radon) transform of a function or vector field. Let $(M,g)$ a non-trapping Riemannian surface-with-boundary with unit circle bundle 
\[ SM = \{(\x,v) \in TM,\ g(v,v) = 1\}, \] 
and let us denote the geodesic flow $\varphi_t:SM \to SM$ by $\varphi_t(\x,v) = (\gamma_{\x,v}(t), \dot\gamma_{\x,v}(t))$, where $\gamma_{\x,v}(t)$ denotes the basepoint and $\dot\gamma_{\x,v}(t)$ is the (unit) velocity vector. For $\x\in \partial M$, let $\nu_\x$ the unit inner normal and define the influx/outflux boundaries $\partial_{\pm} SM = \{(\x,v) \in \partial SM,\ \pm g(\nu_\x,v) > 0\}$. Fix $a$ a smooth enough function on $M$. Then for a function $f\in L^2(SM)$, we define the attenuated (geodesic) X-ray transform of $f$ by 
\begin{align*}
    I_a f (\x,v) := \int_0^{\tau(\x,v)} f(\varphi_t(\x,v)) \exp \left( \int_0^t a(\gamma_{\x,v}(s))\ ds \right) \ dt, \qquad (\x,v)\in \partial_+ SM,
%    I_a F(x,v) &:= \int_0^{\tau(x,v)} g(F(\gamma_{x,v}(t)), \dot\gamma_{x,v}(t))\ dt, \qquad (x,v)\in \partial_+ SM,
\end{align*}
where $\tau(\x,v)$ denotes the first exit time of the geodesic $\gamma_{\x,v}(t)$ (the non-trapping condition implies that $\tau$ is uniformly bounded above on $SM$). Such a transform can be regarded as the influx trace $I_a f = u|_{\partial_+ SM}$ of the solution $u$ to the transport equation on $SM$ 
\begin{align}
    Xu + au = -f \quad (SM), \qquad u|_{\partial_- SM} = 0, 
    \label{eq:attTrans}
\end{align}
and where $X = \frac{d}{dt} \varphi_t |_{t=0}$ is the geodesic vector field. Cases of interest here are $(i)$ the case where $f$ is a function on $M$ i.e. $f(\x,v) = f(\x)$ with applications to X-ray Tomography in media with variable refractive index, a topic receiving much interest at the moment \cite{Nguyen2014,Manjappa2015}, and $(ii)$ the case where $f$ represents a vector field $F$ on $M$ via the relation $f(\x,v) =  g(F(\x), v)$, and whose applications to Doppler tomography justify its nickname of {\em Doppler transform}. 

Such a transform over functions was studied earlier in the Euclidean setting. Inversion methods with known attenuation were obtained independently by Arbuzov, Bukhgeim and Kazantzev using $A$-analytic function theory {\it \`a la} Bukhgeim in 1998 \cite{Arbuzov1998}, and by Novikov via complexification methods in 2002 \cite{Novikov2002,Novikov2002a}, see \cite{Finch2004} for a joint study of both approaches. While both methods present two interesting ways of looking at this problem, the latter is of lesser complexity (i.e., comparable to that of an inverse Radon transform) and is therefore more amenable to implementation, as illustrated in \cite{Natterer2001a,Guillement2002}. This latter method was extended to partial data and more general integrands in \cite{Bal2004a}, to the attenuated Radon transform in hyperbolic geometry and to the horocyclic transform in \cite{Bal2005}, to even more general curves in \cite{Hoell2010}, and more recently, to weighted Radon transforms in the Euclidean plane, where the weight has finite harmonic content in \cite{Novikov2014}, with an implementation in \cite{Guillement2014}. Both methods were also adapted to the study of the attenuated transform over functions and vector fields in {\em fan-beam geometry} in \cite{Kazantsev2007}, describing both a solution via $A$-analytic functions and an approach using fiberwise holomorphic solutions to the transport equation \eqref{eq:attTrans} in the Euclidean setting. This last approach was then generalized by Salo and Uhlmann to simple Riemannian surfaces (i.e. surfaces with strictly convex boundary and no conjugate points) in \cite{Salo2011}, where a method reconstructing functions from knowledge of their geodesic X-Ray transform was developed. The Doppler transform was studied microlocally in the Riemannian case in \cite{Holman2009}, and a range characterization in the Euclidean case was recently given in \cite{Sadiq2014}. Additional range characterizations of the Euclidean transform in convex domains of $\Rm^2$ were provided in \cite{Sadiq2013}, and a study of the attenuated Euclidean transform over second-order tensors was recently provided in \cite{Sadiq2015}. 

% additional generalizations
The X-ray transform can also be considered over vector-valued unknowns defined on $SM$ (i.e. sections of bundles), with matrix-valued attenuations, connections and Higgs fields, for which recent results can be found in \cite{Paternain2013a,Paternain2012,Guillarmou2015}. More general settings include the case of {\em weighted} X-ray transforms (of which the attenuated case is a particular example). In this setting, early Euclidean implementations have been provided in \cite{Kunyansky1992}, and analytic microlocal analysis has led in \cite{Frigyik2008} to injectivity and stability of the attenuated transform over functions when both the metric and the attenuation coefficient are real-analytic. In dimension $d\ge 3$, local injectivity of weighted X-ray transforms near convex boundary points was recently established in \cite{Zhou2013}, following methods in \cite{Uhlmanna} where the unattenuated case was first treated. 
\smallskip

While laying the groundwork of inversions on surfaces with non-constant curvature in \cite{Salo2011} by proving injectivity and providing the first reconstruction procedure, the authors there pointed out two open questions: $(i)$ how to explicitely invert the unattenuated X-ray transform over functions and solenoidal vector fields (we call them $I_0$ and $I_\perp$ here), and $(ii)$ how to explicitely construct {\em (fiber-)holomorphic integrating factors} when curvature is not constant. Here by {\em holomorphic integrating factor} (HIF) for the attenuation $a$, we mean a function which is holomorphic on the fibers (see Sec. \ref{sec:geometrySM}), and which turns the attenuated transport equation \eqref{eq:attTrans} into an unattenuated one. As will be seen below, after performing Fourier analysis on the fibers of $SM$, a transport equation such as \eqref{eq:attTrans} takes the form of a tridiagonal, doubly infinite system of partial differential equations on $M$, which is best understood when it can be made one-sided. As fiber-holomorphic functions have one-sided harmonic content, HIFs allow to fulfill this requirement to a certain extent. In all past aproaches analyzing attenuated transforms, a ``holomorphization'' step of integrating factors always occurs, be it on the fibers of $SM$ when solving a transport equation \cite{Kazantsev2007,Salo2011}, in terms of a complexified parameter when formulating a Riemann-Hilbert problem \cite{Novikov2002,Bal2005,Bal2004a}, or when considering sequence-valued systems \cite{Arbuzov1998,Sadiq2015,Sadiq2014,Sadiq2013}.

In an attempt to address the open questions mentioned above, an answer to $(i)$ was presented by the author in \cite{Monard2013}, by implementing the reconstruction formulas proposed in \cite{Pestov2004,Krishnan2010,Pestov2003}. A first feature of the present article is to address point $(ii)$ by providing explicit ways of constructing holomorphic integrating factors, obtained by first constructing preimages of the adjoint operators $I_0^\star$ and $I_\perp^\star$ explicitely. Second, we derive other reconstruction algorithms, some of which are similar in spirit to those in \cite{Kazantsev2007}, another one similar to the approach in \cite{Pestov2004,Monard2013a}. In the first approach, the main novelty here, partly motivated by the approach in \cite{Kazantsev2007}, consists in decomposing $v = u e^{-w}$ (with $u$ the solution of \eqref{eq:attTrans} and $w$ a holomorphic solution of $Xw = -a$) into the sum of a holomorphic function and a function that is constant along geodesics. This can be done as soon as one can construct explicit preimages of the operators $I_0^\star$ and $I_\perp^\star$. 
A second approach reconstructing functions, more similar to approaches in \cite{Pestov2004,Monard2013a}, is then presented, and a numerical implementation is provided. Additionally, these reconstruction formulas are fast in that no full three-dimensional transport equation needs to be solved unlike \cite{Salo2011}. The class of surfaces where the current approach is valid is that of simple surfaces where some Fredholm equations (see Equations \eqref{eq:FredW} and \eqref{eq:FredWstar} below), which first appeared in \cite[Theorem 5.4]{Pestov2004}, are invertible. It remains open at present whether this latter additional requirement applies to all or a strict subset of simple surfaces, though past numerical experiments done in \cite{Monard2013} by the author have showed that such Fredholm equations were invertible on some family of surfaces which could become arbitrarily close to non-simple. 

%\hl{Connection with other inverse problems ?}  \\
Recent work by the author with P. Stefanov and G. Uhlmann \cite{Monard2013b} shows that stable inversion of the attenuated ray transform should still be possible in some surfaces where conjugate points occur at most in pairs. Adapting the current approach to this latter setting will be the object of future work.

\paragraph{Outline.} The structure of the paper is as follows. We first introduce in Sec. \ref{sec:derivations} the basic setting (Sec. \ref{sec:geometrySM}), some new notation and operators (e.g., $I_\perp$) which play an important role in the inversion process (Sec. \ref{sec:newnotation}), the construction of explicit, continuous right-inverses for $I_0^\star$ and $I_\perp^\star$ (Sec. \ref{sec:I0star}) and how they allow the explicit construction of holomorphic integrating factors (Sec. \ref{sec:holosol}). Section \ref{sec:functions} presents the reconstruction formulas for functions and vector fields, following initial ideas in \cite{Kazantsev2007}, first adapted in \cite{Salo2011}. Section \ref{sec:functions2} proposes an alternative approach to reconstruction of functions, stating in passing $L^2\to L^2$ continuity and bounds for a certain family of operators, the proof of which is relegated to appendix \ref{sec:familyops}. Section \ref{sec:numerics} presents a numerical implementation of inversions as set up in Theorem \ref{thm:iterations}.

\section{Notation, preliminaries and the unattenuated case} \label{sec:derivations}

\subsection{Geometry of the unit circle bundle}\label{sec:geometrySM} 

We briefly recall the geometry and notation associated with the unit circle bundle. With $(M,g)$ as above, the geodesic flow $\varphi_t:SM\to SM$ is defined on the domain 
\begin{align}
    \D = \{(\x,v,t): (\x,v)\in SM, \quad t\in (-\tau(\x, -v), \tau(\x,v))\}.
    \label{eq:D}
\end{align}
with $X$ the geodesic vector field as in the introduction, one may construct a global frame $\{X,X_\perp,V\}$ of $T(SM)$, which encodes the geometry of $M$ via the structure equations
\begin{align*}
    [X,V] = X_\perp, \qquad [X_\perp, V] = X, \qquad [X,X_\perp] = - \kappa V \qquad (\kappa:\text{Gaussian curvature}).
\end{align*}
$V$ is sometimes referred to as the {\em vertical derivative} and $X_\perp$ the {\em horizontal derivative}. For $\x\in M$, we also denote $S_\x := \{v\in T_\x M, \quad g(v,v) = 1\}$ the unit tangent circle at $\x$.

Though the proofs will be coordinate-free, a convenient set of coordinates is that of isothermal coordinates (they can be made global as the simplicity assumption on $M$ implies that it is simply connected), for which the metric $g$ is scalar of the form $g = e^{2\lambda(x,y)}(dx^2 + dy^2)$, and where $(\x,v)\in SM$ is parameterized by $(x,y,\theta)$ where $\x = (x,y)$ and $v = e^{-\lambda(x,y)} (\cos\theta \partial_x + \sin\theta \partial_y)$ for $\theta\in \Sm^1$. In these coordinates, the frame $\{X,X_\perp,V\}$ takes the expression 
\begin{align*}
    X &= e^{-\lambda} (\cos\theta \partial_x + \sin\theta \partial_y + (-\sin\theta \partial_x \lambda + \cos\theta \partial_y \lambda)\ \partial_\theta ),\qquad V = \partial_\theta, \\
    X_\perp &= - e^{-\lambda} (-\sin\theta \partial_x + \cos\theta \partial_y - (\cos\theta \partial_x \lambda + \sin\theta \partial_y \lambda)\ \partial_\theta ).
\end{align*}
We use the Sasaki metric on $SM$ for which the basis $(X,X_\perp,V)$ is orthonormal, with volume form $d\Sigma^3$ preserved by the frame (in isothermal coordinates, $d\Sigma^3 = e^{2\lambda}\ dx\ dy\ d\theta$). Introducing the inner product
\begin{align*}
    (u,v) = \int_{SM} u\bar v\ d\Sigma^3, \qquad u,v:SM\to \Cm,
\end{align*}
the space $L^2(SM,\Cm)$ decomposes orthogonally as a direct sum
\begin{align}
    L^2(SM, \Cm) = \bigoplus_{k\in \Zm} H_k, \qquad\where\quad H_k := \ker (V - ikId). 
    \label{eq:L2decomp}
\end{align}
We also denote $\Omega_k := \C^\infty(SM) \cap H_k$, and denote the corresponding decomposition $u(\x,v) = \sum_{k\in \Zm} u_k(\x,v)$ with $u_k\in H_k$. In isothermal coordinate, this corresponds to Fourier series expansions
\begin{align*}
    u(\x,\theta) = \sum_{k=-\infty}^{\infty} u_k(\x,\theta), \where \qquad u_k(\x,\theta) = e^{ik\theta} \tu_k (\x), \qquad \tu_k(\x) = \frac{1}{2\pi} \int_{\Sm^1} u(\x,\theta) e^{-ik\theta}\ d\theta.
\end{align*}
In particular, we denote either by $u_0$ or $\pi_0 u$ the fiberwise average $u_0 (\x) = \frac{1}{2\pi} \int_{S_\x} u(\x,v)\ dS(v)$, so $\pi_0:L^2(SM)\to L^2(M)$ is the projection onto $H_0$. Such functions admit an even/odd decomposition w.r.t. to the involution $v\mapsto -v$ (or $\theta\mapsto \theta+\pi$), denoted
\begin{align}
    u = u_+ + u_-, \where\quad  u_+ := \sum_{k \text{ even}} u_k \qandq u_- := \sum_{k \text{ odd}} u_k.
    \label{eq:oddeven}
\end{align}
An important decomposition of $X$ and $X_\perp$ due to Guillemin and Kazhdan (see \cite{Guillemin1980}) is given by defining $\eta_{\pm}:= \frac{X\pm iX_\perp}{2}$, so that one has the following decomposition 
\[ X = \eta_+ + \eta_- \qandq X_\perp = \frac{1}{i} (\eta_+ - \eta_-), \]
with the important property that $\eta_{\pm}:\Omega_k\to \Omega_{k\pm 1}$ for any $k\in \Zm$, so that both $X$ and $X_\perp$ map odd functions on $SM$ into even ones and vice-versa. 

In the harmonic decomposition above, a diagonal operator of particular interest is the so-called {\em fiberwise Hilbert transform} $H:\C^\infty(SM)\to \C^\infty(SM)$, whose action on each component is described by 
\begin{align*}
    Hu_k := -i \sgn{k} u_k, \quad k\in \Zm, \quad \text{with the convention }\quad \sgn{0} = 0,
\end{align*}
and we denote $H_{+/-}$ the composition of $H$ with projection onto even/odd Fourier modes. We say that a function $u\in L^2(SM)$ is {\em (fiber-)holomorphic} if $(Id + iH)u = u_0$, i.e. if $u$ has only nonnegative Fourier components. An important identity first proved in \cite{Pestov2005} is the commutator between the Hilbert transform and the geodesic vector field: %denoting $\pi_0:L^2(SM)\to L^2(M)$ the projection onto $H_0$ (fiberwise average)
\begin{align*}
    [H,X] = \pi_0 X_\perp + X_\perp \pi_0, \qquad [H,X_\perp] = -\pi_0 X - X \pi_0. 
\end{align*}
Note also that $H^2 = -Id + \pi_0$ and that $H\pi_0 = \pi_0 H = 0$. Using these observations and the commutators above, we write
\begin{align*}
    [H^2,X] = H[H,X] + [H,X]H = \cancel{H \pi_0 X_\perp} + H X_\perp \pi_0 + \cancel{X_\perp \pi_0 H} + \pi_0 X_\perp H
\end{align*}
On the other hand, $[H^2,X] = [-Id + \pi_0,X] = \pi_0 X - X \pi_0$. Upon splitting into odd and even parts, we arrive at the following equalities, to be used subsequently:
\begin{align}
    \pi_0 X = \pi_0 X_\perp H \qandq X\pi_0 = - HX_\perp \pi_0. 
    \label{eq:ident}
\end{align}

\subsection{An alternative notation for the unattenuated case}\label{sec:newnotation}

We now introduce some notation that emphasizes the $L^2(M)$-duality arising in the Pestov-Uhlmann reconstruction formulas \cite{Pestov2004}. The main novelty below is the introduction of $I_\perp$. The general unattenuated transform can be defined over functions $f\in L^2(SM)$ as 
\begin{align}
    I f(\x,v) = \int_0^{\tau(\x,v)} f(\varphi_t(\x,v))\ dt, \qquad (\x,v)\in \partial_+ SM. 
    \label{eq:gxrt}
\end{align}
Considering integrands of the form $f\in L^2(M)$ (i.e. $f(\x,v) = f(\x)$), and $X_\perp h$ for some $h\in H_0^1(M)$, let us define the unattenuated transforms
\begin{align*}
    I_0 f := If, \qquad I_\perp h := I (X_\perp h),
\end{align*}
(that is, in the definition of $I_0$, we identify $f$ with its pullback by the canonical projection $SM\to M$). These transforms are continuous in the following settings when $(M,g)$ is non-trapping
\[ I_0:L^2(M)\to L^2_\mu (\partial_+ SM), \qquad I_\perp:H_0^1(M) \to L^2_\mu (\partial_+ SM), \] 
where $L^2_\mu$ is a weighted $L^2$ space with weight $\mu(\x,v) = g(\nu_\x,v)$. 
%If the metric is simple, we know that $N = I_0^\star I_0$ is a $\Psi$DO of order $-1$ (\hl{cite}), and if $I_1$. For similar reasons, one would expect that $N_\perp := I_\perp^\star I_\perp$ is a $\Psi$DO of order $1$.
Define the {\em scattering relation} $\alpha:\partial SM \to \partial SM$ as follows: if $(\x,v)\in \partial_+ SM$, $\alpha(\x,v) = \varphi_{\tau(\x,v)}(\x,v) \in \partial_- SM$, and if $(\x,v)\in \partial_- SM$, $\alpha(\x,v) = \varphi_{-\tau(\x,-v)}(\x,v) \in \partial_+SM$. Recall the definitions of $A_{\pm}: L^2_{\mu}(\partial_+ SM) \to L^2_{|\mu|}(\partial SM)$ and their adjoints $A_{\pm}^\star$, introduced in \cite{Pestov2004}:
\begin{align*}
%    A_{\pm} &: L^2_{\mu}(\partial_+ SM) \to L^2_{|\mu|}(\partial SM) \\
    A_{\pm} w(x,\theta) := \left\{
    \begin{array}{c}
	w(x,v) \quad \text{on } \partial_+ SM, \\
	\pm w\circ\alpha (x,v) \quad \text{on } \partial_- SM,
    \end{array}
    \right.  \qquad A_{\pm}^\star u := (u \pm u\circ\alpha)|_{\partial_+ SM}.  
\end{align*}
The fundamental theorem of calculus along a geodesic reads
\begin{align}
    IXu = (u\circ\alpha - u)|_{\partial_+ SM} = - A_-^\star u.
    \label{eq:FTC}
\end{align}
%For $w$ defined on $\partial_+ SM$, we denote by $w_\psi$ the unique solution $u$ to the free transport problem 
%\[ X u = 0 \quad (SM), \qquad u|_{\partial_+ SM} = w. \]
With $\psi: SM \to \partial_- SM$ the endpoint map $\psi(\x,v) = \varphi_{\tau(\x,v)}(\x,v)$ and $h\in L^2_\mu (\partial_+ SM)$, we denote by $h_\psi := h\circ \alpha\circ \psi$ the function extended by free geodesic transport to $SM$, i.e. solution of the equation
\[ X u = 0 \quad (SM), \qquad u|_{\partial_+ SM} = h. \]
Straightforward computations using Santal\'o's formula yield that for $h\in L^2_{\mu} (\partial_+ SM)$, 
\begin{align}
    I_0^\star h = 2\pi (h_\psi)_0, \qquad I_\perp^\star h = - 2\pi (X_\perp h_\psi)_0. 
    \label{eq:adjoints}
\end{align}

\paragraph{The $\V_\pm$ decomposition.} Let us define the {\em antipodal scattering relation} $\alpha_1 : \partial_+ SM\to \partial_+ SM$ as $(v\mapsto -v)\circ \alpha$, and write $L_{\mu}^2(\partial_+ SM)$ as the direct sum $\V_+ \oplus \V_-$, where $h\in \V_+$ (resp. $\V_-$) iff $h$ is even (resp. odd) with respect to the involution $\alpha_1$. Since a function of $\x$ only can be regarded as an even function of $v$ on $SM$, and a vector field can be regarded as an odd function of $v$ on $SM$, it is straighforward to establish that 
\[ \text{Range } I_0\subset \V_+ \qquad \text{and} \quad \text{Range } I_\perp\subset \V_-. \] 
Moreover, we have the following lemma.

\begin{lemma} \label{lem:Vpmdecomp}
    The direct sum $L^2_\mu (\partial_+ SM) = \V_+ \oplus \V_-$ is orthogonal.
\end{lemma}

\begin{proof} For $h\in L^2_\mu (\partial_+ SM)$, define $u = \left(\frac{h}{\sqrt{\tau}}\right)_\psi$ where $\tau = I_0 (\bone)$ ($\bone$ denotes the constant function equal to 1 on $M$). Santal\'o's formula allows to show that the map $h\mapsto u$ is $L^2_\mu (\partial_+ SM)\to L^2(SM)$ continuous and $\|u\|_{L^2(SM)} = \|h\|_{L^2_\mu(\partial_+ SM)}$. Moreover, $u$ is even/odd in $v$ whenever $h\in \V_+/\V_-$, so that, if $h\in \V_+$ and $g\in \V_-$, and upon calling $u = \left( \frac{h}{\sqrt{\tau}} \right)_\psi$ and $v = \left( \frac{g}{\sqrt{\tau}} \right)_\psi$, we have 
    \begin{align*}
	(h,g)_{L_\mu^2(\partial_+ SM)} = (u,v)_{L^2(SM)} = 0,
    \end{align*}
    hence the proof.
\end{proof}

\paragraph{Inversion of $I_0$ and $I_\perp$.} We now revisit the inversion of the operators $I_0$ and $I_\perp$, previously established in \cite{Pestov2004}, adapted here to the present notation. Recall the notation $u^f$ ($f\in L^2(SM)$) for the solution of a transport problem of the form 
\begin{align}
    Xu = -f \quad (SM), \qquad u|_{\partial_- SM} = 0,
    \label{eq:transu}
\end{align}
and for $f\in \C^\infty(M)$, define $Wf = (X_\perp u^f)_0$. It is established in \cite{Pestov2004} that $W$ extends as a smoothing (hence compact) operator $W:L^2(M)\to \C^\infty(M)$ and that the $L^2$-adjoint of $W$ is given by $W^\star h = (u^{X_\perp h})_0$.

\begin{proposition} Let $(M,g)$ a simple surface with boundary. Then we have for every $f\in L^2(M)$ and every $h\in H_0^1(M)$,
    \begin{align}
	f + W^2 f &= \frac{1}{2\pi} I^\star_\perp w, \qquad w = \frac{1}{4} A_+^\star H_- A_- I_0 f, \label{eq:FredW} \\
	h + (W^\star)^2 h &= - \frac{1}{2\pi} I_0^\star w, \qquad w = \frac{1}{4} A_+^\star H_+ A_- I_\perp h.  \label{eq:FredWstar}
    \end{align}    
\end{proposition}

\begin{remark}
    Formulas \eqref{eq:FredW} and \eqref{eq:FredWstar} differ slightly from \cite[Theorem 5.4]{Pestov2004} because it is stated there that the solution to the transport problem 
    \begin{align*}
	Xu = -f, \qquad u|_{\partial_+ SM} = w,
    \end{align*}
    is $u^f + w_\psi$, which is what the Fredholm equation there is based upon. The correct answer would be $u^f + (w-I_0f)_\psi$, which in turn yields the modified formula \eqref{eq:FredW}. 
\end{remark}

\begin{proof} {\bf Proof of \eqref{eq:FredW}.} 
    Let $f\in \C^\infty (M)$ and define $u^f$ as in \eqref{eq:transu} so that $u|_{\partial_+ SM} = I_0 f$. Applying $\pi_0 X = \pi_0 X_\perp H$ (derived in \eqref{eq:ident}) to \eqref{eq:transu}, we obtain 
    \[ f = \pi_0 f = -\pi_0 Xu^f = -\pi_0 X_\perp Hu^f = -\pi_0 X_\perp Hu_-^f = - (X_\perp Hu_-^f)_0, \]
    so $f$ can be obtained from the last equation if we can relate $Hu_-^f$ to the known data $I_0 f$. In order to do so, we use the commutator $[H,X]$ to write a transport equation for $Hu_-^f$
    \begin{align*}
	Xu^f_- = -f \quad\Rightarrow\quad X(Hu_-^f) = - \cancel{X_\perp (u_-^f)_0} - (X_\perp u_-^f)_0 = - Wf,
    \end{align*} 
    so that $Hu_-^f$ satisfies the transport problem 
    \begin{align*}
	X(Hu_-^f) = -Wf, \qquad Hu_-^f|_{\partial_- SM} = \frac{1}{2} H_- A_- I_0 f |_{\partial_- SM} := \eta,
    \end{align*}
    which means that 
    \begin{align*}
	Hu_-^f = u^{Wf} + w_\psi, \quad w = \eta\circ\alpha.
    \end{align*}
    Upon applying $(X_\perp\cdot)_0$, we obtain
    \begin{align*}
	(X_\perp Hu_-^f)_0 = W^2 f - \frac{1}{2\pi} I_\perp^\star w. 
    \end{align*}
    Since we have established that $f = - (X_\perp H u_-^f)_0$, we conclude that 
    \begin{align*}
	f + W^2 f = \frac{1}{2\pi} I^\star_\perp w, \quad w = \left( \frac{1}{2} (H_- A_- I_0 f) |_{\partial_- SM} \right) \circ \alpha.
    \end{align*} 
    In terms of $A_\pm^\star$ operators, we can also write $w$ as $w = \frac{1}{2} \frac{A_+^\star-A_-^\star}{2} H_- A_- I_0 f$.
%    \begin{align*}
%	w = \frac{1}{2} \frac{A_+^\star-A_-^\star}{2} H_- A_- I_0 f.
%    \end{align*}
    Moreover, it can be seen that the function $A_-^\star H_-A_- I_0 f$ has $\V_+$ symmetry, so it is annihilated by $I_\perp^\star$. Thus, Equation \eqref{eq:FredW} follows, extended to every $f\in L^2(M)$ by density of $\C^\infty(M)$ in $L^2(M)$. \smallskip

    {\bf Proof of \eqref{eq:FredWstar}. } Let $h\in \C^\infty(M)$ with $h|_{\partial M} = 0$, and let $u^{X_\perp h}$ solve the transport equation 
    \[ X u = -X_\perp h \quad (SM), \qquad u|_{\partial_- SM} = 0.   \]
    Upon projecting onto odd functions of $v$, we have $X u^{X_\perp h}_+ = - X_\perp h$. Direct manipulations and the use of the commutator formula imply
    \begin{align*}
	X u_+^{X_\perp h} = - X_\perp h \quad \implies\quad X(Hu_+^{X_\perp h} - h) = -X_\perp W^\star h - \cancel{(X_\perp u_+^{X_\perp h})_0} = - X_\perp W^\star h. 
    \end{align*}
    Moreover, the trace of $u_+^{X_\perp h}$ is given by $u_+^{X_\perp h}|_{\partial SM} = \frac{1}{2} A_- I_\perp h$. Since the function $Hu_+^{X_\perp h}$ satisfies the transport problem
    \begin{align*}
	X(Hu_+^{X_\perp h}-h) = -X_\perp W^\star h, \qquad Hu_+^{X_\perp h}|_{\partial_- SM} = \frac{1}{2} H_+ A_- I_\perp h|_{\partial_- SM}, 
    \end{align*}
    we deduce that 
    \begin{align*}
	Hu_+^{X_\perp h}-h = u^{X_\perp W^\star h} + w_\psi, \quad\text{where}\quad w:= \left( \frac{1}{2} H_+ A_- I_\perp h|_{\partial_- SM} \right)\circ \alpha.
    \end{align*}
    Upon averaging over $\theta$ and rearranging terms, we obtain 
    \begin{align*}
	h + (W^\star)^2 h = - \frac{1}{2\pi} I_0^\star w.
    \end{align*} 
    In terms of $A_\pm^\star$ operators, $w$ can also be written as $w = \frac{1}{2} \frac{A_+^\star-A_-^\star}{2} H_+ A_- I_\perp h$. 
%    \begin{align*}
%	w = \frac{1}{2} \frac{A_+^\star-A_-^\star}{2} H_+ A_- I_\perp h.
%    \end{align*}
    We can see that $A_-^\star H_+ A_- I_\perp f$ has $\V_-$ symmetry and as such is annihilated by $I_0^\star$. Thus, Equation \eqref{eq:FredWstar} follows, extended to every $h\in H_0^1(M)$ by density.
\end{proof}

\begin{remark}
    Inspection of symmetries shows that $A_- I_0 f$ in \eqref{eq:FredW} is odd in $v$ and $A_- I_\perp f$ in \eqref{eq:FredWstar} is even in $v$. This tells us that the expressions $H_- A_- I_0 f$ and $H_+ A_- I_\perp f$ are redundant, as one could just replace $H_\pm$ by $H$. This further emphasizes the similarity between formulas \eqref{eq:FredW} and \eqref{eq:FredWstar}, for both of which the operator $A_+^\star H A_-$ acts as a first step in the postprocessing of data, though on a different subspace depending on the formula. 
\end{remark}

\subsection{Construction of explicit continuous right-inverses for $I_0^\star$ and $I_\perp^\star$}
\label{sec:I0star} 

The question of surjectivity of $I_0^\star$ and $I_\perp^\star$ have proved to be crucial for answering boundary rigidity questions (see \cite{Pestov2005} where a proof of surjectivity of $I_0^\star$ appears) and constructing holomorphic integrating factors in a prior study of the attenuated transform (see \cite{Salo2011} where Lemma 4.5 states that $I_\perp^\star$ is surjective), which are so important to the present approach. Such proofs relied on pseudodifferential arguments on an extended simple compact manifold, and did not construct explicit preimages of either operator. In order to derive and implement explicit inversions, constructing explicit preimages becomes a necessity, and we notice here that, while writing formulas \eqref{eq:FredW} and \eqref{eq:FredWstar} in a way that emphasizes duality, we also notice that the right-hand sides involve $I_0^\star$ and $I_\perp^\star$ directly. Under the assumption that the operators $Id + W^2$ and $Id + (W^\star)^2$ are invertible, this allows for an explicit construction of continuous right-inverses of $I_0^\star$ and $I_\perp^\star$.

\begin{remark}\label{rem:closeconstant}
    Although it would be enough to show that $Id + W^2$ is injective, which is open at present for general simple surfaces, it is shown in \cite{Krishnan2010} that the operator $W$ admits a bound of the form $\|W\|_{L^2\to L^2} \le C \|\nabla\kappa\|_\infty$. This implies that if curvature is close enough to constant, the operators $Id + W^2$ and $Id + (W^\star)^2$ are invertible via Neumann series. Whether this qualitative assumption covers the case of all simple surfaces remains open at present.     
\end{remark}

\begin{proposition}\label{prop:rightinverses} Suppose the operators $Id+W^2$ and $Id + (W^\star)^2$ are $L^2(M)\to L^2(M)$ invertible. Then for every $k\in \Nm$, the operators $R_\perp:\C^k(M)\to \C^{k}(\partial_+ SM)$ and $R_0:\C^{k+1}(M)\to \C^k(\partial_+ SM)$ defined by 
    \begin{align}
	R_\perp := \frac{1}{8\pi} A_+^\star H_- A_- I_0 (Id + W^2)^{-1}, \qquad R_0 := -\frac{1}{8\pi} A_+^\star H_+ A_- I_\perp (Id + (W^\star)^2)^{-1},
	\label{eq:R0Rperp}
    \end{align}
    are continuous and satisfy $I_\perp^\star R_\perp f = f$ and $I_0^\star R_0 f = f$ for $f$ smooth enough.
\end{proposition}

\begin{proof} Suppose that $Id + W^2$ and its adjoint are invertible. Then their inverses map any $\C^k(M)$ to itself. This is because the kernels of $W,W^\star$ are proved in \cite{Pestov2004} to be smooth so that, e.g., if $f\in L^2(M)$ solves $f + W^2 f = f_1$, where $f_1\in \C^k(M)$ and $W^2f$ is smooth, then $f = f_1 - W^2 f$ is $\C^k$. Since the operators $W,W^\star$ are also $L^2\to \C^k$ continuous, then similar arguments allow to show that $(Id + W^2)^{-1}$ and its adjoint are $\C^k(M)\to\C^k(M)$ continuous. 

    The relations $I_\perp^\star R_\perp = Id$ and $I_0^\star R_0 = Id$ are straightforward to check, as a directly application of equations \eqref{eq:FredW} and \eqref{eq:FredWstar} and the invertibility of $Id + W^2$ and $Id + (W^\star)^2$.

%    In light of Remark \ref{rem:closeconstant}, when curvature is close enough to constant, we solve for $p,q$ the equations $I_\perp^\star p = f$ and $I_0^\star q = g$ as follows: 
%    \begin{itemize}
%	\item In order to construct $p$, find $f_1$ such that $f_1 + W^2 f_1 = f$ and set $p = \frac{1}{8\pi} A_+^\star H_- A_- I_0 f_1$. 
%	\item In order to construct $q$, find $g_1$ such that $g_1 + (W^\star)^2 g_1 = g$ and set $q = -\frac{1}{8\pi} A_+^\star H_+ A_- I_\perp g_1$. 
%    \end{itemize}
    It remains to prove that 
    \begin{align}
	\|R_\perp f\|_{\C^k (\partial_+ SM)} \le C\|f\|_{\C^k(M)}, \qquad \|R_0 g\|_{\C^{k}(\partial_+ SM)} \le C \|g\|_{\C^{k+1} (M)}. 
	\label{eq:contpq}
    \end{align}
    Looking at the compound expression of these operators, we see that $A_-$ and $A_+^\star$ preserve $\C^k$ norms since the scattering map is smooth and the function $\tau(\x,v)$ is smooth on $\partial_+ SM$ whenever $\partial M$ is strictly convex (see \cite[Lemma 4.1.1 p.115]{Sharafudtinov1994}), $H_\pm$ preserve $\C^k$ norms as convolution operators, and $I_0: \C^k(M)\to \C^k (\partial_+ SM)$ and $I_\perp:\C^{k+1}(M) \to \C^k(\partial_+ SM)$ are continuous since the geodesic flow is smooth.
\end{proof}

\begin{remark}
    A study of symmetries with respect to the involution $\alpha_1$ shows that $p = R_\perp f$ and $q=R_0 g$ thus constructed satisfy $p\in \V_-$ and $q\in \V_+$. This is compliant with the continuity statements \eqref{eq:contpq}, as any component of $p$ in $\V_+$ would be annihilated by $I_\perp^\star$ and any component of $q$ in $\V_-$ would be annihilated by $I_0^\star$. 
\end{remark}

\subsection{Holomorphic solutions to certain transport equations}\label{sec:holosol}

A crucial tool in the inversion of attenuated ray transforms is the construction of {\em holomorphic integrating factors}, whose existence relies on the surjectivity of $I_0^\star$ and $I_\perp^\star$. In the simple Riemannian setting, it is proved in \cite[Theorem 4.1]{Paternain2012} that the transport equation $Xu = -F$ (for $F\in \C^\infty(SM)$) admits holomorphic solutions if and only if $F$ is of the form $F = f_1 + X_\perp f_2$ for some functions $f_1,f_2\in \C^\infty(M)$. Although uniqueness of such solutions may not hold (e.g. adding a constant to such a solution makes another solution), a constructive approach, inspired in part by \cite[Theorem 4.1]{Paternain2012}, is to look for an ansatz, holomorphic by construction, of the form 
\begin{align*}
    u = (Id + iH) p_\psi + (Id + iH) q_\psi,
\end{align*}
where $p$ is a smooth element in $\V_-$ and $q$ is a smooth element in $\V_+$, so that $p_\psi$ is odd and $q_\psi$ is even. Plugging this ansatz into $Xu = -f_1-X_\perp f_2$, we obtain
\begin{align*}
    Xu &= X (Id+iH) p_\psi + X(Id+iH)q_\psi \\
    &= - i [H,X] p_\psi - i [H,X]q_\psi \qquad \qquad (Xp_\psi = Xq_\psi = 0)  \\
    &= -i (X_\perp p_\psi)_0 - i \cancel{X_\perp (p_\psi)_0} - i \cancel{(X_\perp q_\psi)_0} - i X_\perp (q_\psi)_0 \\
    &= \frac{i}{2\pi} I_\perp^\star p - \frac{i}{2\pi} X_\perp I_0^\star q \qquad \qquad (\text{using } \eqref{eq:adjoints}). 
\end{align*} 
Therefore, a sufficient condition for $u$ to solve $Xu=-F$ is if $p$ and $q$ satisfy 
\begin{align*}
    \frac{-i}{2\pi} I^\star_\perp p = f_1, \qandq \frac{i}{2\pi} I_0^\star q = f_2,
\end{align*}
which we may solve explicitly using the previous section. Using Proposition \ref{prop:rightinverses}, we summarize this construction in the following result, whose proof is straightforward and omitted. 

\begin{proposition}\label{prop:holosol}
    Under the hypotheses of Proposition \ref{prop:rightinverses}, and given $k\in \Nm$, $f_1\in \C^k(M)$ and $f_2\in\C^{k+1} (M)$, the function 
    \begin{align*}
	u = 2\pi i [(Id + iH) (R_\perp f_1)_\psi - (Id + iH) (R_0 f_2)_\psi],
    \end{align*}
    is a holomorphic solution of $Xu = -f_1 - X_\perp f_2$ on $SM$, satisfying the estimate
    \begin{align*}
	\|u\|_{\C^k(SM)} \le C (\|f_1\|_{\C^k(M)} + \|f_2\|_{\C^{k+1}(M)}),
    \end{align*}
    where $R_0$ and $R_\perp$ are defined in \eqref{eq:R0Rperp}.
\end{proposition}

%\hl{Write a proposition.}
%Proposition \ref{prop:rightinverses} shows that when $Id + W^2$ and $Id + (W^\star)^2$ are invertible, both equations above can be solved, with continuous dependence on $f_1,f_2$ as in \eqref{eq:contpq}. This means that if $f_1\in \C^k(M)$ and $f_2\in\C^{k+1} (M)$ for some $k\in \Nm$, then $u$ belongs to $\C^k(M)$ and there exists some constant $C$ independent of $u,f_1,f_2$ such that  
%\begin{align*}
%    \|u\|_{\C^k(M)} \le C (\|f_1\|_{\C^k(M)} + \|f_2\|_{\C^{k+1}(M)}).
%\end{align*}

\section{Inversion of the attenuated ray transform over functions and vector fields {\em \`a la} Kazantzev-Bukhgeim}\label{sec:functions}

In \cite{Kazantsev2007}, the authors provide reconstruction formulas for functions and vector fields from knowledge of their ray transforms in the case where the metric is Euclidean and the domain is the unit disk. The present section generalizes these ideas to the case of simple Riemannian surfaces. 
\subsection{Reconstruction of a function}

In this first approach, we follow the idea in \cite{Kazantsev2007} that, if we can find a solution $u^\star$, holomorphic with $u^\star_0 = 0$ of $X u^\star + au^\star = -f$, then projecting this equation onto $\Omega_0$ gives 
\[ f = - \eta_+ u^\star_{-1} - \eta_- u^\star_1 - a u^\star_0 = - \eta_- u^\star_1, \] 
after which one must explain how to express $u^\star_1$ in terms of known data. Here and below, assuming that the operators $Id+W^2$ and $Id+ (W^\star)^2$ are invertible, we denote $I_{0,\perp}$ the ray transform restricted to integrands of the form $f_1 + X_\perp f_2$, where $f_1,f_2\in \C^\infty(M)$. If $D:= I_0 f_1 + I_\perp f_2 = I (f_1 + X_\perp f_2)$, we can reconstruct $f_1, f_2$ (and in turn, $f_1 + X_\perp f_2$) from $D$ by carrying out the following steps: 
\begin{enumerate}
    \item Compute the $\V_+/\V_-$ decomposition of $D$. This decomposition will be given by $\frac{1}{2} (D\pm \alpha_1^\star D)\in \V_\pm$, with $\alpha_1$ the antipodal scattering relation. 
    \item Reconstruct $f_1$ from the projection of $D$ onto $\V_+$ by inverting \eqref{eq:FredW}. 
    \item Reconstruct $f_2$ from the projection of $D$ onto $\V_-$ by inverting \eqref{eq:FredWstar}. 
\end{enumerate}
In what follows, we summarize the above procedure by writing $I_{0,\perp}^{-1} D = f_1 + X_\perp f_2$. 
 
\begin{theorem} \label{thm:frc1}
    Let $(M,g)$ a simple Riemannian surface with boundary and $f, a \in \C^\infty(M, \Rm)$. Then $f$ is uniquely determined by its attenuated geodesic transform via the reconstruction formula
    \begin{align}
	f = 2i \eta_- (\Im (e^{w}h'_\psi) )_1, \qquad h' := \frac{1}{2} (D - w')|_{\partial_+ SM},
	\label{eq:frc1}
    \end{align}
    where $w$ is an odd, holomorphic solution of $Xw = -a$, $D = ((Id - iH) (e^{-w} (u|_{\partial SM}))_-$ with $u|_{\partial_- SM} = 0$ and $u|_{\partial_+ SM} = I_a f$, and $w'$ is a holomorphic solution of $Xw' = - I_{0,\perp}^{-1} (A_-^\star D)$, given by Proposition \ref{prop:holosol}.
\end{theorem}

\begin{proof}
  {\bf Step 1: find $u^\star$ a holomorphic solution of $Xu^\star + a u^\star = -f$.}
  Call $u$ the solution to $Xu + au = -f$, $u|_{\partial_- SM} = 0$ so that $u|_{\partial_+ SM} = I_a f$. Let $w$ a holomorphic, odd solution of $Xw = -a$. Then $v = e^{-w} u$ solves the transport problem
    \begin{align}
	Xv = e^{-w} (Xu - (Xw) u) = e^{-w} (Xu+au) = - e^{-w} f, \qquad v|_{\partial_+ SM} = e^{-w} I_a f, \qquad v|_{\partial_- SM} = 0,
	\label{eq:transv}
    \end{align}
    The right-hand side $b := e^{-w} f$ is holomorphic. This is because, since the product of holomorphic functions is holomorphic, so are (convergent) powers series of holomorphic functions. Moreover, plugging the expansion $w = w_1 + w_3 + \dots$ ($w_j\in\Omega_j$) into the exponential, we see that the fiberwise decomposition of $b(\x,v)$ looks like
    \begin{align*}
	b = f (1 + w + \frac{w^2}{2} + \dots) = \underbrace{f}_{b_0} + \underbrace{f w_1}_{b_1} + \underbrace{f \frac{w_1^2}{2}}_{b_2} + \dots, 
    \end{align*}
    so $b_0 = f$. We now want to decompose $v$ into $v = v^\star + h'_\psi$, where $v^\star$ is holomorphic and $h'_\psi$ is constant along geodesics, then we will have that $v^\star$ solves $X v^\star = -b$. We proceed as follows. First decompose $v = \frac{1}{2} (v^{(+)} + v^{(-)})$, where $v^{(\pm)} = (Id \pm iH) v$. The function $v^{(-)}$ solves the transport equation 
    \begin{align*}
	X v^{(-)} &= X (v - iHv) \\
	&= (Id - iH) Xv + i [H,X] v \\
	&= -(Id-iH) b + i (X_\perp v)_0 + i X_\perp v_0 \\
	&= - (f - i (X_\perp v)_0) + i X_\perp v_0.
    \end{align*}
    Integrating along geodesics, we deduce that 
    \begin{align*}
	I_0 (f - i (X_\perp v)_0) - iI_\perp v_0 = A_-^\star v^{(-)}|_{\partial SM}, 
    \end{align*}
    where the right-hand-side $A_-^\star (v^{(-)}|_{\partial SM}) = A_-^\star D$, where $D$ has the expression in the statement of the theorem and is known from data $I_a f$. Therefore the right-hand side $(f - i (X_\perp v)_0) - i X_\perp v_0$ can be reconstructed upon inverting $I_0$ and $I_\perp$, a relation which we denote 
    \begin{align*}
	(f - i (X_\perp v)_0) - i X_\perp v_0 = I_{0,\perp}^{-1} A_-^\star D.	
    \end{align*}    
    %Setting $g:= (f - i (X_\perp v)_0)$, we have that $I_0 g$ and $iI_\perp v_0$ are known since $v^{(-)}$ is known on $\partial SM$, so $g$ and $v_0$ can be reconstructed via inverting $I_0$ and $I_\perp$, respectively. 
    Let $w'$ a second holomorphic function such that $Xw' = - (f - i (X_\perp v)_0) + i X_\perp v_0$, constructed following Proposition \ref{prop:holosol}, that is, $w' = 2\pi i (Id + iH) (p_\psi + q_\psi)$ with $p,q$ solving 
    \begin{align*}
	I_\perp^\star p = (f - i (X_\perp v)_0) \qandq I_0^\star q = iv_0.	
    \end{align*}   
    With $w'$ thus constructed, we have $X (v^{(-)} - w') = 0$, i.e. $(v^{(-)} - w')$ is constant along geodesics. Upon rewriting $v$ as $v = \frac{1}{2} (v^{(+)} + w') + \frac{1}{2}(v^{(-)} - w')$, we see that the first term is holomorphic and the second is constant along geodesics. In other words, $v$ is of the form $v = v^\star + h'_\psi$, where $v^\star = \frac{1}{2} (v^{(+)} + w')$ and $h'_\psi = \frac{1}{2}(v^{(-)} - w')$. Additionally, with this choice of $w'$, we have
    \begin{align*}
	v^\star_0 = \frac{1}{2} (v^{(+)}_0 + w'_0) = \frac{1}{2} (v_0 + 2\pi i(q_\psi)_0) = \frac{1}{2} (v_0 + iI_0^\star q) = 0.
    \end{align*}
    Finally, defining $u^\star = v^\star e^w$, we see that $u^\star$ is holomorphic as the product of holomorphic functions, and, using the last equation, we see that $u_0^\star = (v^\star e^w)_0 = v_0^\star = 0$, and that it solves
    \begin{align*}
	X u^\star + a u^\star = - b e^w = -f.
    \end{align*}
     Projecting the equation above into $\Omega_0$, we obtain 
     \begin{align*}
       -f = \eta_+ u^\star_{-1} + \eta_- u^\star_1 + a u^\star_0 \quad \implies\quad f = -\eta_- u_1^\star, 
     \end{align*}
     so it remains to show how to compute $u^\star_1$ in terms of known data. 
    
     {\bf Step 2: express $u^\star_1$ in terms of known data.} We now write
     \begin{align*}
       u_1^\star = u_1^\star - \overline{u_{-1}^\star} = u_1^\star - (\overline{u^\star})_1 = 2i(\Im (u^\star))_1 = 2i(\Im(u^\star-u))_1, 
     \end{align*}
     where the first equality comes from the fact that $u_{-1}^\star=0$ and the last comes from the fact that $u$ is real valued. We now write, by definition, 
    \begin{align*}
	u^\star - u = (v^\star - v) e^w = - e^w h'_\psi,
    \end{align*}
    so that $\Im (u^\star - u) = - \Im (e^w h'_\psi)$, and we arrive at the expression
    \begin{align*}
	u_1^\star = -2i (\Im (e^w h'_\psi))_1.
    \end{align*}
    Now looking to compute $h'$, since we proved that $h'_\psi = \frac{1}{2} (v^{(-)} - w')$, then 
    \begin{align}
	\begin{split}
	    h' &= \frac{1}{2} (v^{(-)} - w')|_{\partial_+ SM}  \\
	    &= \frac{1}{2} v^{(-)}|_{\partial _+ SM} - \frac{1}{2} (2\pi i) [(Id + iH) (p_\psi + q_\psi)]|_{\partial _+ SM}, \\
	    &\text{where } \qquad I^\star_\perp p = g, \qquad I_0^\star q = i v_0, \quad I_0 g - i I_\perp v_0 = A_-^\star ( v^{(-)}|_{\partial SM} ). 	    
	\end{split}
	\label{eq:hprime}
    \end{align}
    The proof is complete. 
\end{proof}

\begin{remark} This proof generalizes the one completed in \cite{Kazantsev2007} in the Euclidean case when the domain is a disk. In order to complete the argument there, it is required to relate the fiberwise Hilbert transform with the Hilbert transform of the domain for the so-called divergent beam transform (see \cite[Lemma 4.1]{Kazantsev2007}), which in turn uses the singular value decomposition (SVD) of that operator, established in earlier references (see \cite{Kazantsev2007} for detail and references there). 

    While this SVD, specific to the choice of metric and domain, does not seem straightforward to generalize systematically, the present construction of holomorphic solutions allows here to simplify the proof by bypassing these steps altogether.     
\end{remark}

In isothermal coordinates $(\x,\theta)$, Equation \eqref{eq:frc1} takes the following form:
\begin{align}
    f = -\eta_- u^\star_1 = - e^{-2\lambda} \overline\partial \left( \widetilde{u_1^\star}\ e^\lambda\right) , \qquad \widetilde{u_1^\star}(\x) = \frac{-2i}{2\pi} \int_0^{2\pi} e^{-i\theta} \Im (e^{w(\x,\theta)}h'_\psi(\x,\theta))\ d\theta,
    \label{eq:fiso}
\end{align}
with $\overline\partial:= \frac{1}{2} (\partial_x + i\partial_y)$. Constructing both holomorphic solutions $w,w'$ using Proposition \ref{prop:holosol}, we arrive at the following reconstruction procedure: 
\begin{enumerate}
  \item Construct $w = 2\pi i (I+iH) (R_\perp a)_\psi$, odd holomorphic solution of $Xw = -a$. 
  \item Compute $v|_{\partial_+ SM} = e^{w} I_a f$ and $v|_{\partial_- SM} = 0$, then $v^{(-)}|_{\partial SM} = (Id - iH) (v|_{\partial SM})$. 
  \item Reconstruct $g$ and $v_0$ from  $I_0 g - i I_\perp v_0 = A_-^\star ( v^{(-)}|_{\partial SM} )$. (inversion of $I_0$ and $I_\perp$)
  \item Construct $p$ and $q$ from $i I_\perp^\star p = g$ and $I_0^\star q = i v_0$. (inversion of $I_0^\star$ and $I_\perp^\star$, see Sec. \ref{sec:holosol})
  \item Construct $h'$ from \eqref{eq:hprime}.
  \item Reconstruct $\widetilde{u^\star_1}$ then $f$ according to \eqref{eq:fiso}.
\end{enumerate}

\subsection{Reconstruction of vector fields}

As in \cite{Kazantsev2007}, the method of proof above can easily be generalized to the case of reconstruction of vector fields, which we now present. Such a problem has applications to Doppler tomography in media with variable index of refraction, previously studied in \cite{Holman2009,Kazantsev2007}. Integrands of the form $f(\x,\theta) = f_0(\x) + f_1(\x) \cos(\theta-\alpha)$ are also considered in \cite{Bal2004a} in the Euclidean setting. In our context, a smooth real-valued vector field takes the form $\V(\x,v) = (F_1 + F_{-1})/2$, with $F_{\pm 1}$ an element of $\Omega_{\pm 1}$ and $\overline{F_1} = F_{-1}$. In isothermal coordinates, we may write $F_1(\x,\theta) = (f_1(\x) - i f_2(\x))e^{-\lambda(\x)} e^{i\theta}$ for some real-valued functions $f_1,f_2$, and this would correspond to integrating the vector field $f_1(\x) \partial_x + f_2(\x)\partial_y$.

\begin{theorem}\label{thm:Doppler}
    Let $(M,g)$ a simple Riemannian surface with boundary and $\V (\x,v) = (F_1 + F_{-1})/2$ a smooth vector field as above with $F_{-1} = \overline{F_1}$. Then at every $\x\in M$ where $a(\x)\ne 0$, $F_{-1}$ (and hence $\V$) can be uniquely reconstructed from data $I_a \V$ via the formula
    \begin{align}
	F_{-1} = -4i \eta_- \left( \frac{1}{a} \eta_- (\Im(e^w h'_\psi))_1 \right), \qquad h' := \frac{1}{2} (D-w')|_{\partial_+ SM}, 
	\label{eq:vrc}
    \end{align}
    where $w$ is an odd, holomorphic solution of $Xw = -a$, $D = ((Id-iH) (e^{-w} (u|_{\partial SM}))_-)|_{\partial SM}$ with $u|_{\partial_+ SM} = I_a \V$ and $u|_{\partial_- SM} = 0$, and $w'$ is a holomorphic solution of $Xw' = -I_{0,\perp}^{-1} (A_-^\star D)$.    
\end{theorem}

\begin{proof}
    Start from the equation
    \begin{align*}
	Xu + au = - (F_1 + F_{-1})/2, \qquad u|_{\partial_+ SM} = I_a \V, \quad u|_{\partial_- SM} = 0.
    \end{align*}

    \smallskip
    {\bf Step 1: find $u^\star$ a holomorphic solution of $Xu^\star + au^\star = -\V$.} Let $w = w_1 + w_3 + \dots$ be an odd, holomorphic solution of $Xw = -a$ and define $v = e^{-w} u$. Then $v$ satisfies the transport problem
    \begin{align*}
	Xv = - e^{-w} (F_1 + F_{-1})/2 = -G(\x,v) = - (G_{-1} + G_0 + G_1 + G_2 \dots),
    \end{align*}
    where $G_{-1}(\x,v) = F_{-1}(\x,v)/2$ and $G_0(\x) = -w_1 F_{-1}/2$. Split $v$ into an holomorphic and antiholomorphic part, i.e. look at $v = \frac{1}{2} (v^{(+)} + v^{(-)})$, where $v^{(\pm)} = (Id \pm H) v$. $v^{(-)}$ satisfies the transport equation
    \begin{align*}
	X v^{(-)} = X(v - iHv) &= (Id - iH) Xv + i [H,X]v \\
	&= - (Id - iH)G + i(X_\perp v)_0 + i X_\perp v_0 \\
	&= - (G_0 - i(X_\perp v)_0) + iX_\perp v_0 - G_{-1}.  
    \end{align*} 
    The data $D$ defined in the statement is $D = v^{(-)}|_{\partial SM}$. Using the Hodge decomposition, we now write $G_{-1} = X g + X_\perp h$ for $g,h$ two functions on $M$, where $g$ fulfills the additional prescription $g|_{\partial M} = 0$. Then the previous transport equation becomes
    \begin{align*}
	X (v^{(-)} + g) = - (G_0 - i(X_\perp v)_0) + X_\perp (iv_0 - h). 
    \end{align*}
    Note that $(v^{(-)} + g)|_{\partial_+ SM} = v^{(-)}|_{\partial_+ SM} = (Id - iH) (e^{-w}I_a f)|_{\partial_+ SM}$ is known, so that upon integrating the transport equation along each geodesic we can see that the data gives us 
    \begin{align*}
	I_0 (G_0 - i(X_\perp v)_0) - I_\perp (iv_0 - h) = A_-^\star D,
    \end{align*}    
    with $D$ defined as in the statement of the theorem. Now construct $w'$ a holomorphic solution of 
    \[Xw' = - (G_0 - i(X_\perp v)_0) + X_\perp (iv_0 - h) = - I_{0,\perp}^{-1} [A_-^\star D], \]
    so that $X (v^{(-)} + g - w') = 0$, i.e. there exists $h$ defined on $\partial_+ SM$ such that $v^{(-)} + g - w' = h_\psi$.
    
    We now decompose $v = \frac{1}{2} (v^{(+)} + w' - g) + \frac{1}{2} (v^{(-)} - w' + g) = v^\star + h'_\psi$, where the first term is holomorphic and the second is constant along geodesics. In particular, we get that $X v^\star = - G(x,v)$ where $v^\star$ is now holomorphic. Then defining $u^\star = e^w v^\star$, we find that $u^\star$ is holomorphic and satisfies
    \[ X u^\star + a u^\star = -(F_{-1} + F_1)/2. \]
    Looking at the projections onto $\Omega_{-1}$ and $\Omega_0$ yields the relations
    \[ \eta_- u_0^\star = -F_{-1}/2, \quad \eta_- u_1^\star + au_0^\star = 0,    \]
    which implies the reconstruction formula, at each point where $a$ does not vanish:
    \begin{align*}
	F_{-1} =  2 \eta_- \left( \frac{1}{a} \eta_- u_1^\star \right). 
    \end{align*}
    
    {\bf Step 2: obtain $u_1^\star$ from the measurements.} This part is, again, similar to the proof of Theorem \ref{thm:frc1}. We write
     \begin{align*}
       u_1^\star = u_1^\star - \overline{u_{-1}^\star} = u_1^\star - (\overline{u^\star})_1 = 2i(\Im (u^\star))_1 = 2i(\Im(u^\star-u))_1, 
     \end{align*}
     where the first equality comes from the fact that $u_{-1}^\star=0$ and the last comes from the fact that $u$ is real valued. Next we have the relation
     \[u^\star-u = (v^\star-v) e^{w} = - e^w h'_\psi, \]
    so it remains to compute $h'_\psi$, which after unrolling definitions, 
    \begin{align*}
	h' = h'_\psi|_{\partial_+ SM} = \frac{1}{2} (v^{(-)} - w' + g)|_{\partial_+ SM} = \frac{1}{2} (v^{(-)} - w')|_{\partial_+ SM},
    \end{align*}
    where both terms are, again, computible from data: $v^{(-)}|_{\partial_+ SM} = D|_{\partial_+ SM}$ and $w'$ is a holomorphic solution of $Xw' = - I_{0,\perp}^{-1} [A_-^\star D]$, a solution of which can be explicitely constructed following Proposition \ref{prop:holosol}. This ends the proof.  
\end{proof}

\section{A second approximate formula for functions, conditionally invertible via Neumann series}\label{sec:functions2}

While Theorem \ref{thm:frc1} reconstructs functions exactly and, in some sense, in a ``one-shot'' fashion, it presents a couple of weaknesses: $(i)$ it involves all values of $w(\x,\theta)$ throughout $SM$, which in turn involves storing three dimensions of data when everything should be dealt with using two-dimensional structures, and $(ii)$ the effects of curvature (which need iterative corrections as in \cite{Monard2013}) are not being corrected in the right place. 

We now propose an algorithm based on a more direct interplay of the Hilbert transform with transport equations as in \cite{Monard2013a,Pestov2004}, which leads to a Neumann-series based inversion, faster in implementation, and valid when curvature is close enough to constant and the attenuation is small enough in $\C^2$ norm. Such an algorithm is then implemented in Section \ref{sec:numerics}. 

We first state a result about a certain family of operators generalizing the operator $Wf = (X_\perp u^f)_0$ first defined in \cite{Pestov2004}. These operators appear as error operators of the next reconstruction formula. We relegate the proof and some remarks about these operators to the appendix. 

\begin{proposition}\label{prop:familyops}
    Let $h\in L^\infty(SM)$ such that $Vh\in \C^1(SM)$. Then the operator $K_h:L^2(M)\to L^2(M)$ defined by $K_h f := (X_\perp u^{fh})_0$ is well-defined and continuous, and there exists a constant $C$ such that 
    \begin{align}
	\|K_h\|_{L^2\to L^2} \le C (\|\nabla \kappa\|_\infty \|h\|_\infty + \|Vh\|_{\C^1}).
	\label{eq:Kest}
    \end{align}
\end{proposition}

We now state the main result of the section. 

\begin{theorem}[Iterative inversion]\label{thm:iterations}
    If $f\in L^2(M)$ and $I_a f$ denotes its attenuated GXRT with given attenuation $a\in \C^2(M)$, then the function $f$ satisfies the following equation
    \begin{align}
	f + Kf = \frac{1}{2\pi} I_\perp^\star \eta, \qquad \eta = \frac{1}{4} A_+^\star H (e^{-w} I_a f)_-, 
	\label{eq:fNeumann}
    \end{align}
    where $(e^{-w} I_a f)_-$ denotes extension of $e^{-w} I_a f$ from $\partial_+ SM$ to $\partial SM$ by oddness w.r.t. $v\mapsto -v$, $w$ is an odd, holomorphic function solving $Xw = -a$ and $K:L^2(M)\to L^2(M)$ is a bounded operator satisfying the following estimate
    \begin{align}
	\|K\|_{L^2\to L^2} \le C \left( \|\nabla \kappa\|_\infty^2 + e^{\|a\|_\infty} (\|a\|_{\infty} \|\nabla\kappa \|_\infty + \|a\|_{\C^2}) \right).
	\label{eq:Kbound}
    \end{align}
\end{theorem}

\begin{proof} {\bf Proof of \eqref{eq:fNeumann}. }
    Let $w$ a holomorphic, odd solution to $Xw = -a$. If $u$ is the solution to $Xu + au = -f$ with boundary condition $u|_{\partial_- SM} = 0$, then the function $v = e^{-w}u$ solves the transport problem $Xv = -b := -fe^{-w}$ with boundary condition $v|_{\partial_- SM} = 0$, so that $v$ is no other than $u^{fe^{-w}}$. Applying $\pi_0 X = \pi_0 X_\perp H$ (derived in \eqref{eq:ident}) to the equation $Xv = -b$, we obtain
    \begin{align*}
	-f = -\pi_0 b = \pi_0 Xv = (X_\perp Hv)_0 = (X_\perp Hv_-)_0. 
    \end{align*}
    The task now is to make $Hv_-$ more explicit. Hitting the transport equation satisfied by $v$ with the Hilbert transform $H$, we write a transport equation for $Hv$
    \begin{align*}
	X(Hv) = - H b - X_\perp v_0 - (X_\perp v)_0.
    \end{align*}
    As $b$ is holomorphic, we have $Hb = f H(e^{-w}) = -i f (e^{-w}-1)$ and upon taking the even part of the last equation w.r.t. $\theta\mapsto \theta+\pi$, the function $Hv_{-}$ solves the transport equation
    \begin{align*}
	X(H v_-) = i f (\cosh w - 1) - (X_\perp v)_0,
    \end{align*}
    where we have used that $w(\x,\theta+\pi) = -w(\x,\theta)$. This tells us that 
    \begin{align*}
	H v_- = u^{-i f(\cosh w - 1)} + u^{(X_\perp v)_0} + \eta_\psi,
    \end{align*}
    where $\eta_\psi$ is constant along geodesics. We can deduce $\eta$ from the fact that the first two terms in the last r.h.s. vanish on $\partial_- SM$, so that $\eta = (Hv_-|_{\partial_- SM}) \circ \alpha$, and since $v_-$ is known from data on $\partial SM$, so is $\eta$. More precisely, we have, for any $(\x,v)\in \partial SM$
    \begin{align}
	v_-(\x,v) = \frac{1}{2} (e^{-w} I_a f)_- (\x,v) := \left\{ \begin{array}{ll}
	    \frac{1}{2} e^{-w(\x,v)} I_a f(\x,v) & \text{if }(\x,v)\in \partial_+ SM, \\
	    \frac{-1}{2} e^{-w(\x,-v)} I_a f (\x,-v) & \text{if }(\x,v) \in \partial_- SM.		
	    \end{array}
	    \right.
	    \label{eq:vminus}
    \end{align}
    Upon applying the operator $\pi_0 X_\perp$, we obtain 
    \begin{align*}
	f &= -(X_\perp Hv_-)_0 \\
	&= i (X_\perp u^{f(\cosh w - 1)})_0 - W (X_\perp v)_0 - (X_\perp \eta_\psi)_0 \\
	&= i (X_\perp u^{f(\cosh w - 1)})_0 - W (X_\perp u^{fe^{-w}})_0 - (X_\perp \eta_\psi)_0
    \end{align*}
    i.e. this equation takes the form
    \begin{align}
	\begin{split}
	    f + K f &= -(X_\perp \eta_\psi)_0, \qquad \eta = (Hv_-)|_{\partial_- SM}\circ\alpha, \\
	    \text{where } \quad K f &:= W (X_\perp u^{fe^{-w}})_0 - i (X_\perp u^{f(\cosh w - 1)})_0.
	\end{split}    
	\label{eq:frc2}
    \end{align}
    Note that upon rewriting $u^{fe^{-w}} = u^f + u^{f(e^{-w}-1)}$, the operator $K$ can be rewritten as 
    \begin{align*}
	Kf &= W(X_\perp u^f)_0 + W (X_\perp u^{f(e^{-w}-1)})_0 - i(X_\perp u^{f(\cosh w -1)})_0 \\
	&= W^2 f + W (X_\perp u^{f(e^{-w}-1)})_0 - i(X_\perp u^{f(\cosh w -1)})_0.
    \end{align*}
    Equation \eqref{eq:frc2} now corresponds to \eqref{eq:fNeumann} upon noticing that $I_\perp^\star \eta = -2\pi (X_\perp \eta_\psi)_0$, and that 
    \begin{align*}
	I_\perp^\star \eta = I_\perp^\star \left( \frac{1}{2} (H (e^{-w} I_a f)_-)|_{\partial_- SM}\circ\alpha \right) = I_\perp^\star \left( \frac{1}{2} \frac{A_+^\star - A_-^\star}{2} (H (e^{-w} I_a f)_-)\right) = \frac{1}{4} I_\perp^\star A_+^\star H (e^{-w} I_a f)_-,
    \end{align*}
    where the last equality comes from the fact that $A_-^\star$ applied to an odd function on $\partial SM$ makes a function with $\V_+$ symmetry, which in turn is annihilated by $I_\perp^\star$. 
    
    {\bf Proof of \eqref{eq:Kbound}.} The theorem will be proved once we show that the operator $K$ satisfies the bound \eqref{eq:Kbound}. From the last equation, we see that $K = W^2 + W K_{h_1} - i K_{h_2}$, where $h_1 = e^{-w}-1$ and $h_2 = \cosh w - 1$ and the notation $K_{h_j}$ refers to Proposition \ref{prop:familyops}. It is proved in \cite{Krishnan2010} that $\|W\|_{L^2\to L^2} \le C\|\nabla\kappa\|_\infty$ for some constant $C$. By virtue of Proposition \ref{prop:familyops}, we deduce that for $j=1,2$ 
    \[ \|K_{h_j}\|_{L^2 \to L^2} \le C (\|\nabla\kappa\|_\infty \|h_j\|_\infty + \|V h_j\|_{\C^1}).    \]
    Now we bound $\|h_j\|_\infty \le \|w\|_\infty e^{\|w\|_\infty}$ and $\|Vh_j\|_{\C^1} \le C \|Vw\|_{\C^1} e^{\|w\|_\infty} \le C\|w\|_{\C^2} e^{\|w\|_\infty}$ for $j=1,2$, and together with the fact that $w$ is constructed following Prop. \ref{prop:holosol}, it satisfies estimates of the form 
    \begin{align*}
	\|w\|_{\C^k} \le C \|a\|_{\C^k}, \qquad k=0,1,2,\dots
    \end{align*}
    Combining all these estimates together, we obtain estimate \eqref{eq:Kbound}. 
\end{proof}

\begin{remark}
    In the absence of attenuation $a\equiv 0$, equation \eqref{eq:fNeumann} is exactly the Fredholm equation \eqref{eq:FredW}, in which case $w(\x,v)\equiv 0$ so that $Kf = W^2 f$ and the right-hand side of \eqref{eq:frc2} is the post-processing of unattenuated data.
\end{remark}

\begin{corollary}\label{cor:fNeumann}
    Under the hypotheses of Theorem \ref{thm:iterations}, if $\|\nabla\kappa\|_\infty$ and $\|a\|_{\C^2}$ are small enough so that estimate \eqref{eq:Kbound} implies $\|K\|_{L^2\to L^2}<1$, then $f$ can be reconstructed from equation \eqref{eq:fNeumann} via the Neumann series
    \begin{align*}
	f = \frac{1}{2\pi} \sum_{n=0}^\infty (-K)^n I_\perp^\star \eta, \qquad \eta = \frac{1}{4} A_+^\star H (e^{-w} I_a f)_-.
    \end{align*}
\end{corollary}

\begin{remark} As in the unattenuated case, this restriction on $\|\nabla\kappa\|_\infty$ and $\|a\|_{\C^2}$ is of rather qualitative nature and does not tell us whether all simple cases will work, and whether this approach would work for attenuations less than $\C^2$. This is to be contrasted with successful numerical reconstructions below, which work for both discontinuous attenuations and cases of metrics arbitrarily close to non-simple.    
\end{remark}

\section{Numerical implementation}\label{sec:numerics}

We now present a brief implementation of an inversion of \eqref{eq:fNeumann} via a Neumann series approach. We use the code developed by the author in \cite{Monard2013} for the unattenuated case, augmented with attenuation. The domain $M$ is the unit disk $\{\x = (x,y), x^2 + y^2\le 1\}$ endowed with the metric $g(\x) = e^{2\lambda(\x)} Id$, where 
\begin{align*}
    \lambda(\x) = 0.2 \left( \exp\left( \frac{(\x - \x_0)^2}{2\sigma^2}\right) - \exp\left( \frac{(\x + \x_0)^2}{2\sigma^2}  \right)  \right), \qquad \x_0 = (0.3,0.3), \qquad \sigma = 0.25,
\end{align*}
describing a region of ``low sound speed'' near $\x_0$ and ``high sound speed'' near $-\x_0$. The effect on geodesic curves can be seen Fig. \ref{fig:geodesics}. For such a domain, it can be computed that the boundary is strictly convex and there are no conjugate points. 

The influx boundary $\partial_+ SM = \Sm^1\times (-\frac{\pi}{2},\frac{\pi}{2})$ is parameterized by $\x(\beta) = \binom{\cos\beta}{\sin\beta}$ and ingoing speed direction $\theta(\beta,\alpha) = \beta + \pi + \alpha$ (in this case, $\beta+\pi$ is the direction of the unit inner normal). $M$ is represented into the unit square $[-1,1]^2$, discretized in an equispaced cartesian fashion with $N = 300$ points along each dimension. 

\begin{figure}[htpb]
    \centering
    \includegraphics[trim= 50 0 60 0, clip, height = 0.195\textheight]{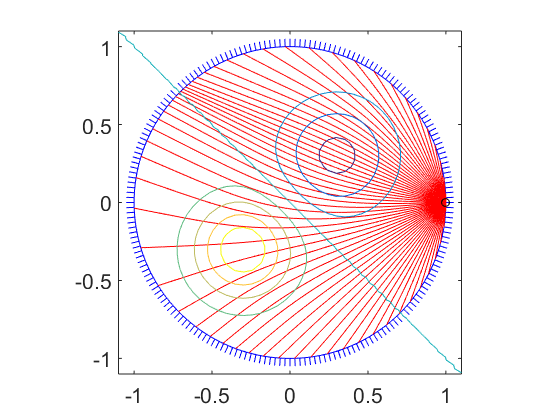}
    \includegraphics[trim= 40 0 60 0, clip, height = 0.195\textheight]{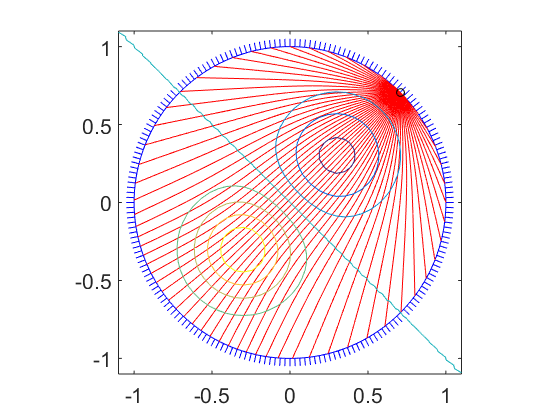}
    \includegraphics[trim= 40 0 60 0, clip, height = 0.195\textheight]{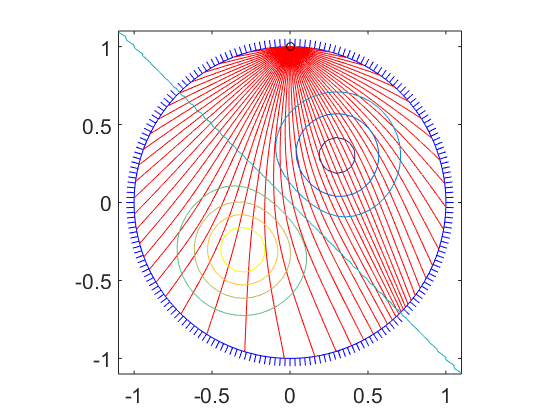}
    \includegraphics[trim= 40 0 40 0, clip, height = 0.195\textheight]{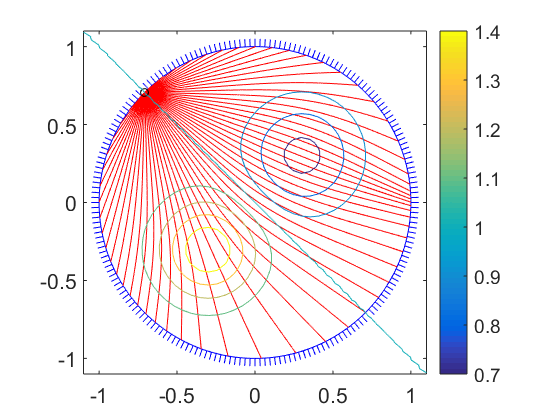}
    \caption{Geodesics cast from the boundary point $\x(\beta)= \binom{\cos\beta}{\sin\beta}$ with, from left to right, $\beta = 0, \frac{\pi}{4}, \frac{\pi}{2}, \frac{3\pi}{4}$. The colorbar on the right indicates the values of the background sound speed $c(\x) = e^{-\lambda(\x)}$.}
    \label{fig:geodesics}
\end{figure}

The function $f$ and the attenuation $a$ are displayed on Fig. \ref{fig:function}. Note that both quantities contain jump singularities. 

\begin{figure}[htpb]
    \centering
    \includegraphics[trim= 30 10 20 10, clip, height = 0.205\textheight]{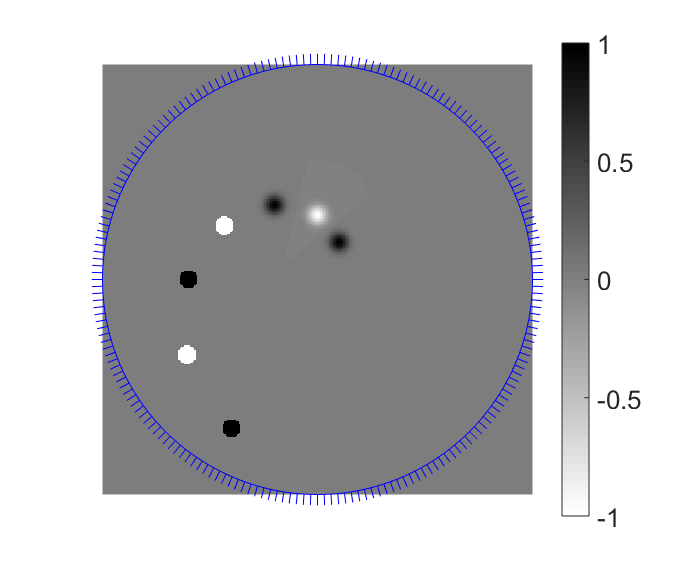}
    \includegraphics[trim= 30 10 20 10, clip, height = 0.205\textheight]{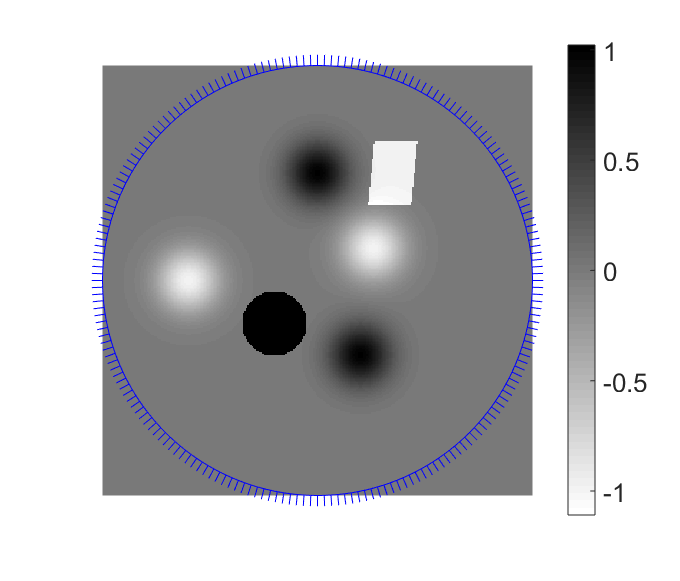}
    \includegraphics[trim= 30 10 20 10, clip, height = 0.205\textheight]{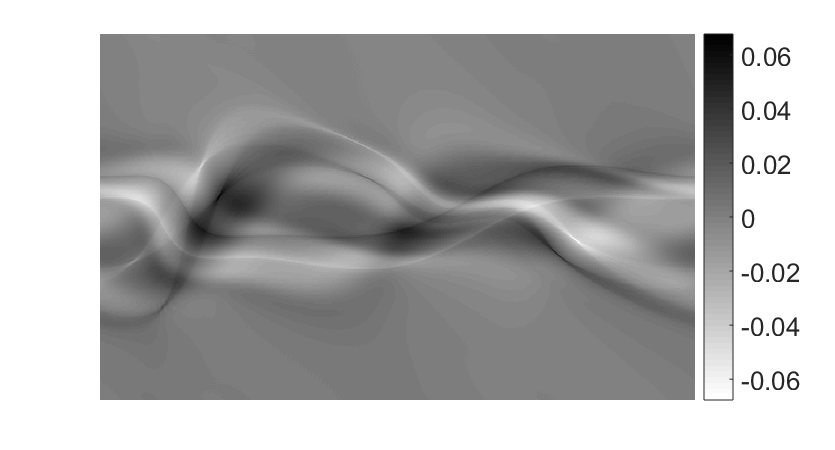}
    \caption{From left to right: Unknown $f$, attenuation $a$, function $n$ defined on $\partial_+ SM$ such that $I_\perp^\star n = a$ (so that $w = 2\pi i(Id + iH) n_\psi$ is a holomorphic solution of $Xw = -a$). Axes for $n$ are $(\beta,\alpha)\in [0,2\pi]\times \left( \frac{-\pi}{2},\frac{\pi}{2}  \right)$.}
    \label{fig:function}
\end{figure}

\paragraph{Strategy for inversion.} Equation \eqref{eq:fNeumann} is of the form
\begin{align}
    f + Kf = L_a I_a f, 
    \label{eq:frc_numerics}
\end{align}
where $I_a f$ represents forward data, $L_a D = \frac{1}{8\pi} I_\perp^\star A_+^\star H (e^{-w} D)_-$ is an approximate inversion and we assume that $K$ is a contraction. Once $I_a$ and $L_a$ are discretized (call their discretized versions with the same name for simplicity), discretizing $K$ separately and computing a finite sum of $\sum_{k=0}^\infty (-K)^k L_a I_a f$ to reconstruct $f$ may introduce additional numerical errors due to the fact that \eqref{eq:frc_numerics} may not be satisfied at the numerical level. A better approach is to set directly $-K = Id - L_a I_a$ and to implement a finite sum of the series 
\begin{align*}
    f = \sum_{k=0}^\infty (Id - L_a I_a)^k L_a I_a f.
\end{align*}
We now briefly explain how the operators $I_a$ and $L_a$ are implemented. 

\paragraph{Forward operator $I_a$.}

The computation of the forward data is done by discretizing the influx boundary $\partial_+ SM$ into an equispaced family $(\beta_i,\alpha_j)_{1\le i\le 2N, 1\le j\le N}$ and for each data point, we compute $I_a f(\beta_i,\alpha_j)$ by discretizing the system of ODEs over $t\in [0,\tau_{ij}]$ where $\tau_{ij} = \tau(\x(\beta_i),\theta(\beta_i,\alpha_j))$
\begin{align*}
    \dot\x = e^{-\lambda(\x)} \binom{\cos\theta}{\sin\theta} , \qquad \dot\theta = e^{-\lambda(\x)} (\partial_y \lambda(\x) \cos\theta - \partial_x\lambda (\x) \sin\theta), %\qquad \dot u = - a(\x) u + f(\x), %[\nabla\lambda(\x), \hat\theta]
\end{align*}
with initial conditions $\x(0) = \binom{\cos\beta_i}{\sin\beta_i}$, $\theta(0) = \beta_i + \pi + \alpha_j$, along which we compute 
\[ I_a f (\beta_i, \alpha_j) = \int_0^{\tau_{ij}} f( \x_{\beta_i,\alpha_j}(t)) \exp\left( \int_0^t a(\x_{\beta_i,\alpha_j} (s))\ ds\right)\ dt,\] 
via a discrete sum, with $\x_{\beta_i, \alpha_j}$ the solution of the ODE above. 
%with initial conditions $\x(0) = \x(\beta_i)$, $\theta(0) = \theta(\beta_i,\alpha_j)$ and $u(0) = 0$, and setting $I_a f(\beta_i,\alpha_j) = u(\tau_{ij})$.

\paragraph{Computation of a holomorphic solution $w$ of $Xw = -a$.} Following Proposition \ref{prop:holosol}, we look for $w$ in the form $w = 2\pi i(Id + iH) n_\psi$, where $n$ solves $I_\perp^\star n = a$. This requires implementing $n = R_\perp a$, with $R_\perp = \frac{1}{8\pi} A_+^\star H_- A_- I_0 (Id + W^2)^{-1}$. Following \cite{Monard2013}, $(Id + W^2)^{-1}$ is computed via a few iteration of a rapidly convergent Neumann series, $I_0$ is the unattenuated X-ray transform. In the present case, $n$ is represented on the right-hand plot of Figure \ref{fig:function}. Once $n$ is computed, as the expression of $L_a$ only involves values of $w$ on $\partial SM$, then we can compute
\begin{align*}
    w|_{\partial SM} = (2\pi i (Id + iH) n_\psi) |_{\partial SM} = 2\pi i (Id + iH) (n_\psi) |_{\partial SM}  =  2\pi i (Id + iH) A_+ n,
\end{align*}
where the Hilbert transform is processed via Fast Fourier Transform on the columns of $A_+ n$. As $n$ has $\V_-$ symmetry, $A_+ n$ amounts to extending $n$ to $\partial_- SM$ by oddness w.r.t. $\theta\mapsto \theta+\pi$ (the latter is much more straightforward numerically).

\paragraph{Approximate inverse $L_a$.}
For $D$ a data function defined on a discretization of $\partial_+ SM$, we wish to compute $L_a D = \frac{1}{8\pi} I_\perp^\star A_+^\star H (e^{-w} D)_-$. The computation of $w$ and $H$ is explained in the previous paragraph. $A_+^\star h(\beta,\alpha)$ is computed by combining values of $h(\beta,\alpha)$ and $h$ at the endpoint of the geodesic starting from coordinate $(\beta,\alpha)$. The main technical step is the computation of $I_\perp^\star$, which in isothermal coordinates can be simplified as follows (see \cite[Sec. 3.1.2]{Monard2013}):
\begin{align*}
    I_\perp^\star h(\x) = e^{-2\lambda(\x)} \nabla_\x \cdot \left( e^{\lambda(\cdot)} \int_{\Sm^1} \binom{-\sin\theta}{\cos\theta} h_\psi(\cdot,\theta)\ d\theta \right),
\end{align*}
where $\nabla_\x\cdot$ is just divergence on a cartesian grid. Integrals in $\Sm^1$ can then be discretized using finite sums and the divergence is implemented using finite differences.

Both computations of $A_+^\star$ and $I_\perp^\star$ require computing several geodesic endpoints, which is the main bottleneck of the code. An alternative option, trading memory for much shorter CPU time, is to compute and store all endpoints required at first, and reusing them in further Neumann iterations.

\begin{figure}[htpb]
    \centering
    \includegraphics[trim= 30 20 10 20, clip, height = 0.21\textheight]{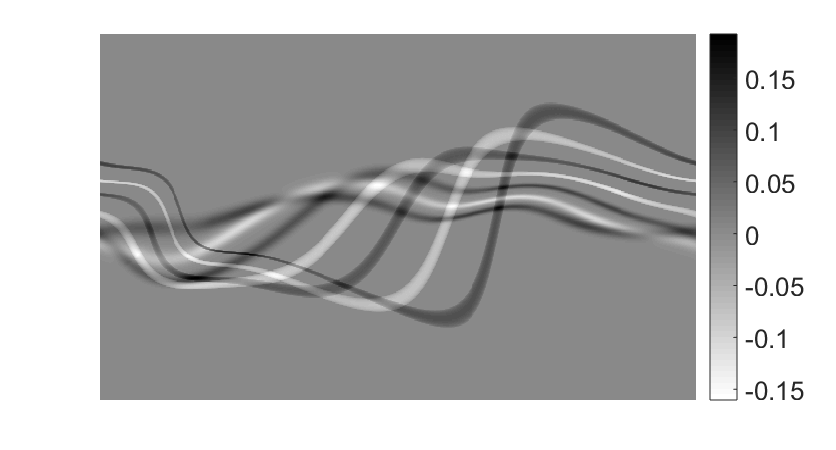}
    \includegraphics[trim= 30 20 10 20, clip, height = 0.21\textheight]{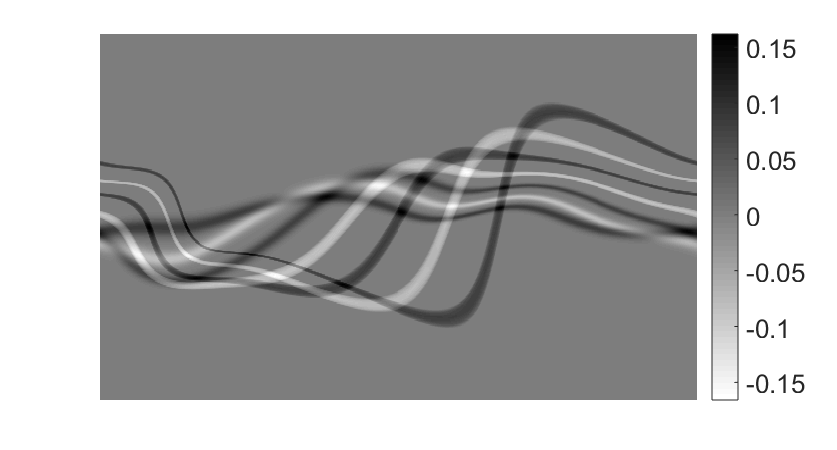}
    \caption{Left: attenuated data $I_a f$ (data for Experiment 1) with $f,a$ from Fig. \ref{fig:function}. Right: unattenuated data $I_0 f$ for comparison. The differences may be observed in magnitude but not in support. }
    \label{fig:data}
\end{figure}

%\begin{figure}[htpb]
%    \centering
%    \includegraphics[trim= 10 20 0 50, clip, height = 0.21\textheight]{realdata}
%    \includegraphics[trim= 10 20 0 50, clip, height = 0.21\textheight]{imagdata}
%    \caption{Real (left) and imaginary (right) parts of $e^{-w} I_a f$}
%    \label{fig:dataHIF}
%\end{figure}

\paragraph{Experiments.} We present two experiments, in which $a$ and $f$ refer to the functions displayed on Figure \ref{fig:function}. $5a$ refers to the function $M\ni\x\mapsto 5 a(\x)$.

\begin{description}
    \item[$\bullet$ Experiment 1.] (low attenuation) Neumann series based reconstruction of $f$ from $I_a f$. 
    \item[$\bullet$ Experiment 2.] (high attenuation) Neumann series based reconstruction of $f$ from $I_{5a} f$. 
\end{description}
Experiment 1 successfully and stably reconstructs $f$ within 3 Neumann iterations, as shown in Fig. \ref{fig:errorExp1} (up to numerical inaccuracies, and given the fact that jumps can never be fully captured exactly). As this example works even if $a$ is discontinuous, this is to be contrasted with the regularity requirements on $a$ from Theorem \ref{thm:iterations}.

\begin{figure}[htpb]
    \centering
    \includegraphics[trim= 30 30 10 50, clip, height = 0.21\textheight]{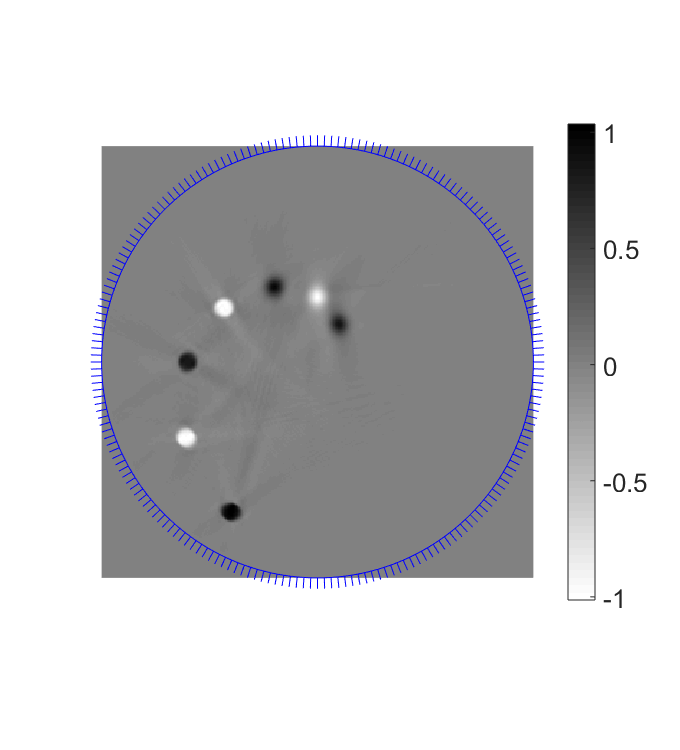}
    \includegraphics[trim= 30 30 10 50, clip, height = 0.21\textheight]{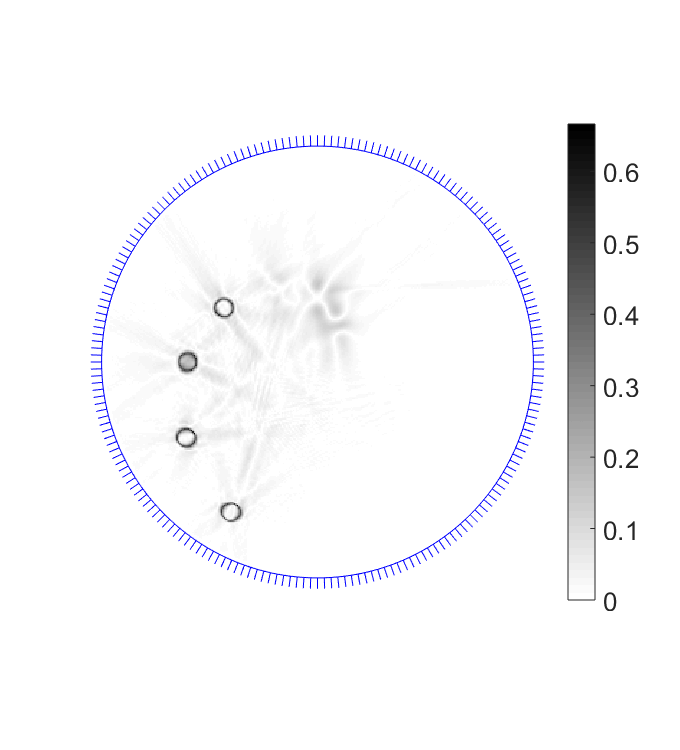}
    \includegraphics[trim= 30 30 10 50, clip, height = 0.21\textheight]{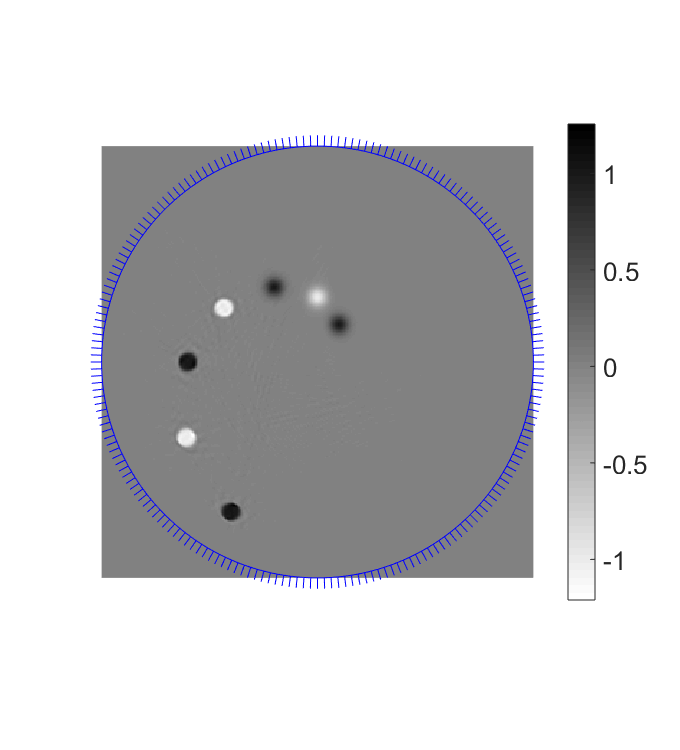}
    \includegraphics[trim= 30 30 10 50, clip, height = 0.21\textheight]{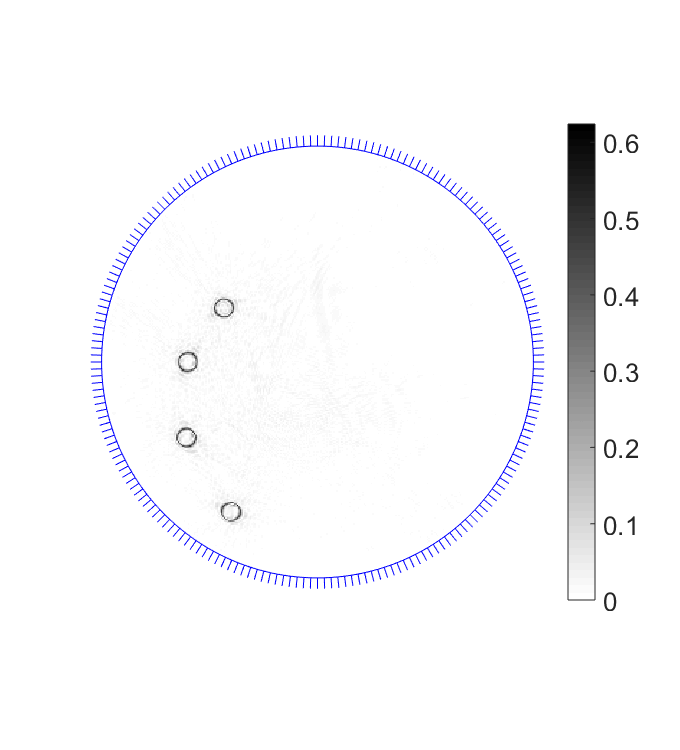}
    \caption{From left to right: reconstruction $f_{rc}$ from $I_a f$ and pointwise error $|f_{rc}-f|$ after one iteration, $f_{rc}$ from $I_a f$ and pointwise error $|f_{rc}-f|$ after three iterations.}
    \label{fig:errorExp1}
\end{figure}

Experiment 2 displays a divergent Neumann series, due to the fact that attenuation $5a$ is too high. This is in agreement with the smallness requirements of Corollary \ref{cor:fNeumann} on the attenuation coefficient, as the operator norm of the error operator in \eqref{eq:Kbound} potentially grows like $e^{\|a\|_\infty}$. 

\begin{figure}[htpb]
    \centering
    \includegraphics[trim= 60 20 10 20, clip, height = 0.21\textheight]{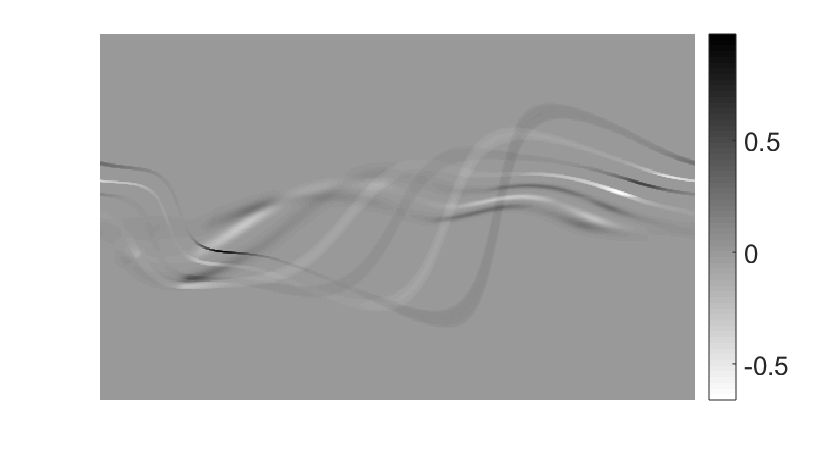}
    \includegraphics[trim= 30 40 20 60, clip, height = 0.21\textheight]{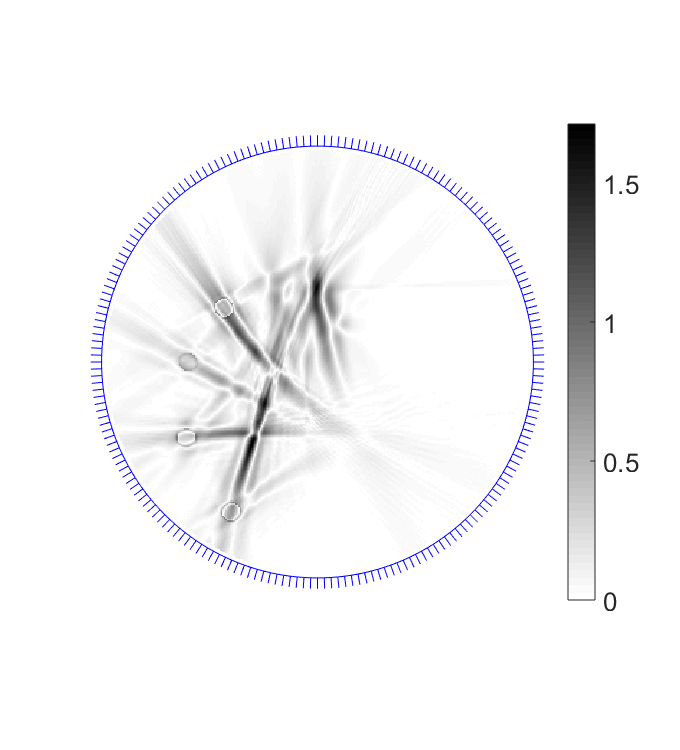}
    \includegraphics[trim= 30 40 20 60, clip, height = 0.21\textheight]{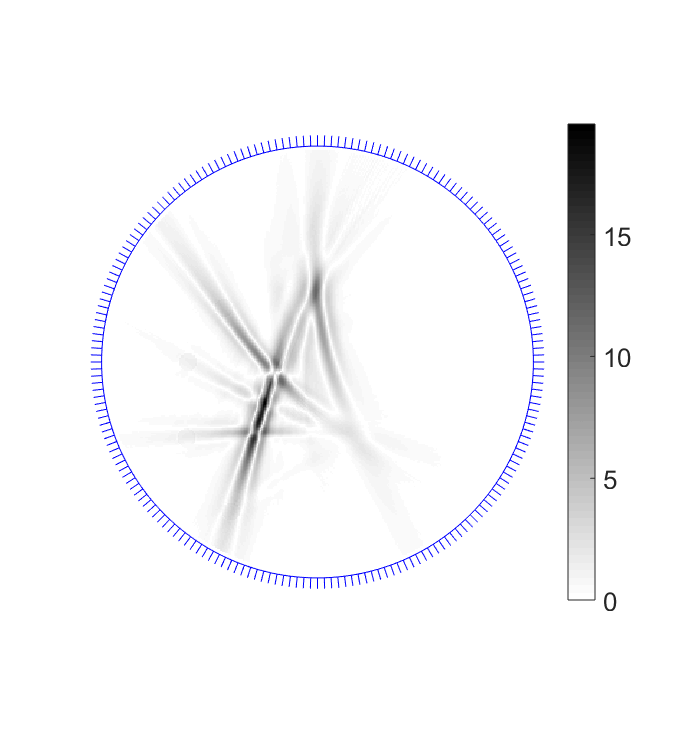}
    \caption{From left to right: $I_{5a} f$ (data for Experiment 2), pointwise error $|f_{rc}-f|$ after one iteration, pointwise error $|f_{rc}-f|$ after two iterations. Divergence of the algorithm ensues due to the appearing unstable modes.}
    \label{fig:Exp2}
\end{figure}

\section*{Acknowledgements}

The author thanks Gunther Uhlmann for encouragement and support, Hart Smith, Colin Guillarmou and Plamen Stefanov for helpful discussions, Larry Pierce for communicating references \cite{Nguyen2014,Manjappa2015}, as well as the anonymous referees for valuable comments. Partial funding by NSF grant No. 1265958 is acknowledged.

\appendix

\section{A certain family of operators - proof of Proposition \ref{prop:familyops}}\label{sec:familyops}

On simple surfaces, it is proved in \cite{Pestov2004} that the operator $Wf = (X_\perp u^f)_0$ is smoothing on simple surfaces, so that the equation reconstructing a function from its unattenuated ray transform is of Fredholm type.

It has been observed that if $f$ is now a function on $SM$, the corresponding operator no longer has such properties. However, the operators $K_h$ introduced in Proposition \ref{prop:familyops} appear naturally as error operators in Theorem \ref{thm:iterations}, and they generalize $W$ since $W = K_h$ with $h\equiv 1$. In general, $K_h$ may no longer be smoothing, but we can still obtain $L^2(M)\to L^2(M)$ continuity with estimates on the norm in terms of $h$ and the ambient curvature. We must recall some facts about Jacobi fields (or variations of the exponential map), following \cite{Merry2011}. For $\xi\in T_{(\x,v)} SM$, we may decompose $d\varphi_t(\xi)$ along the frame $\{X(t),X_\perp(t), V(t)\}$ at the basepoint $\varphi_t(\x,v)$ as
\begin{align*}
  d\varphi_t(\xi) = \zeta_1 (\x,v,t) X(t) + \zeta_2 (\x,v,t) X_\perp (t) + \zeta_3 (\x,v,t) V(t). 
\end{align*}
The structure equations then provide us a differential system in $t$ for the coefficients $\zeta_i$:
\begin{align*}
    \dot \zeta_1 = 0, \qquad \dot \zeta_2 + \zeta_3 = 0, \qquad \dot \zeta_3 - \kappa(\gamma_{\x,v}(t)) \zeta_2 = 0.
\end{align*}
In particular, we may express the variation fields $d\varphi_t(X_\perp) = a X_\perp(t) - \dot a V(t)$ and $d\varphi_t(V) = -b X_\perp(t) + \dot b V(t)$ in term of two functions $a$, $b$ defined on $\D$, solving
\begin{align*}
    \ddot a + \kappa(\gamma_{\x,v}(t)) a = \ddot b + \kappa(\gamma_{\x,v}(t)) b = 0, \qquad \left[
    \begin{smallmatrix}
	 a & b \\ \dot a & \dot b
    \end{smallmatrix}
\right](0) = \left[
\begin{smallmatrix}
    1 & 0 \\ 0 & 1
\end{smallmatrix}
\right].
\end{align*} 
The assumption of simplicity implies that $b$ does not vanish outside $\{t=0\}$ and is thus positive for all $t>0$. Note the constancy of the Wronskian $a\dot b - b \dot a \equiv 1$. 

\begin{proof}[Proof of Proposition \ref{prop:familyops}] For a smooth function $\phi(\x,v)$ on $SM$ vanishing on $\partial SM$, we first write, using the chain rule,
    \begin{align*}
	X_\perp \int_0^{\tau(\x,v)} \phi (\varphi_t(\x,v))\ dt &= \int_0^{\tau(\x,v)} (a(\x,v,t) X_\perp \phi(\varphi_t(\x,v))  - \dot a (\x,v,t) V\phi(\varphi_t(\x,v)) )\ dt.
    \end{align*}
    We then rewrite (keeping $(\x,v)$ implicit), for $t\ne 0$: 
    \begin{align*}
	a X_\perp \phi - \dot a V \phi = \frac{a}{b} (b X_\perp \phi - \dot b V\phi) + \left( -\dot a + \frac{a \dot b}{b} \right) V\phi = - \frac{a}{b} V(\phi \circ \varphi_t) + \frac{1}{b} V\phi. 
    \end{align*}
    Integrating this equality for $t\in (0,\tau(\x,v))$ and integrating by parts over $S_\x$ (using that $\phi$ vanishes at $\partial SM$), we arrive at the conclusion that
    \begin{align*}
	\int_{S_\x} X_\perp \int_0^{\tau(\x,v)} \phi (\varphi_t(\x,v))\ dt\ dS(v) = \int_{\Sm^\x} \int_0^{\tau(\x,v)} \left( V \left( \frac{a}{b} \right) \phi + \frac{1}{b} V\phi \right)\ dt\ dS(v).
    \end{align*}

%\begin{align*}
%    \int_{\Sm^1} X_\perp \int_0^{\tau(\x,\theta)} \phi (\gamma_{\x,\theta}(t), \alpha_{\x,\theta}(t))\ dt\ d\theta &= \int_{\Sm^1} \int_0^{\tau(\x,\theta)} (X_\perp \gamma\cdot \nabla_\x \phi +  (X_\perp \alpha)V\phi )\ dt\ d\theta \\
%    &= \int_{\Sm^1} \int_0^{\tau(\x,\theta)} \left( \frac{a}{b} \partial_\theta\gamma\cdot\nabla_\x \phi + (X_\perp \alpha) V\phi \right) \\
%    &= \int_{\Sm^1} \int_0^{\tau(\x,\theta)} \left( \frac{a}{b} \partial_\theta \phi + (X_\perp \alpha - \frac{a}{b} \partial_\theta \alpha) V\phi \right)\ dt\ d\theta \\
%    &= \int_{\Sm^1} \int_0^{\tau(\x,\theta)} \left( -\partial_\theta \left( \frac{a}{b} \right) \phi + \frac{1}{b} V\phi \right)\ dt\ d\theta,
%\end{align*} 
%where we used at the third step that $X_\perp \alpha - \frac{a}{b} \partial_\theta \alpha = \frac{1}{b}$, which follows from \cite[Lemma 5.5]{Monard2013a} and the fact that the Wronskian $a \dot b - b \dot a \equiv 1$. 

Now replacing $\phi(\x,v)$ by $f(\x) h(\x,v)$ and using the fact that $V(fh) = fV(h)$, we arrive at 
\begin{align*}
    K_h(f)(\x) &= \frac{1}{2\pi} \int_{S_\x} \int_0^{\tau(\x,v)} \left( V \left( \frac{a}{b} \right) h + \frac{1}{b} Vh \right) f (\gamma_{\x,v}(t))\ dt\ dS(v) \\
    &= K_{h,1} f(\x) + K_{h,2} f(\x),
\end{align*}
upon expanding the sum. The operator $K_{h,1}$ is just as well-behaved as the operator $W$ and for the same reason: defining $q(\x,v,t) := \frac{1}{b (\x,v,t)} V\left( \frac{a(\x,v,t)}{b(\x,v,t)}\right)$, it is shown in \cite{Krishnan2010} that $|q(\x,v,t)| \le C \|\nabla \kappa\|_\infty$ for every $(\x,v,t)\in \D$. So we can rewrite 
\begin{align*}
    K_{h,1}f(\x) &= \frac{1}{2\pi} \int_{S_\x} \int_0^{\tau(\x,v)} q(\x,v,t) h(\varphi_t(\x,v)) f(\gamma_{\x,v}(t))\ b(\x,v,t)\ dt\ dS(v) \\
    &= \frac{1}{2\pi} \int_{M} q(\x,v(\y),t(\y)) h(\varphi_{t(\y)}(\x,v(\y)) f(\y)\ dM_\y,
\end{align*}
so that the kernel $k_{h,1}(\x,\y) = q(\x,v(\y),t(\y)) h(\varphi_{t(\y)}(\x,v(\y))$ of $K_{h,1}$ is bounded (hence in $L^2(M\times M)$), i.e. the operator $K_{h,1}:L^2(M)\to L^2(M)$ is continuous (in fact, compact) with an operator norm controlled by $C\|\nabla \kappa\|_\infty \|h\|_\infty$. On to the study of the second term
\begin{align*}
    K_{h,2} f(\x) = \frac{1}{2\pi} \int_{S_\x} \int_0^{\tau(\x,v)} \frac{q_2(\x,v,t)}{b^2(\x,v,t)} f(\gamma_{\x,v}(t))\ b\ dt\ dS(v), \qquad q_2(\x,v,t) := Vh(\varphi_t(\x,v)).
\end{align*}
The function $q_2$ satisfies $\int_{S_\x} q_2(\x,v,0)\ dS(v) = \int_{S_\x} Vh(\x,v)\ dS(v) = 0$, so that the integral is expected to make sense as a principal value integral. 
Note that near $t=0$, we have $b(\x,v,t) = t +  t^3 c(\x,v,t)$ where $c$ is smooth on $\D$, and $1 +  t^2 c(\x,v,t)$ does not vanish on $\D$ since $b$ does not vanish outside $\{t=0\}$ by simplicity of the surface. More precisely, we write
\begin{align*}
    2\pi K_{h,2} f(\x) = \int_{S_\x} \int_0^{\tau(\x,v)} \frac{q_2(\x,v,t)}{b^2(\x,v,t)} f(\gamma_{\x,v}(t))\ b \ dt\ dS(v) = Af(\x) + Bf(\x) + Cf(\x)
\end{align*}
where, upon writing $b(\x,v,t) = t + t^3 c(\x,v,t)$, we define  
\begin{align*}
    Af(\x) &= \int_{S_\x} \int_0^{\tau(\x,v)} \frac{q_2(\x,v,0)}{t^2} f(\gamma_{\x,v}(t))\ b \ dt\ dS(v) \\
    Bf(\x) &= \int_{S_\x} \int_0^{\tau(\x,v)} \frac{q_2(\x,v,t)-q_2(\x,v,0)}{t^2} f(\gamma_{\x,v}(t))\ b\ dt\ dS(v) \\
    Cf(\x) &= - \int_{S_\x} \int_0^{\tau(\x,v)} q_2(\x,v,t) \frac{c(\x,v,t)(2+t^2c(\x,v,t))}{(1+t^2c(\x,v,t))^2} f(\gamma_{\x,v}(t))\ b\ dt\ dS(v). 
\end{align*}
Upon changing variable $(v,t)\mapsto \y(v,t) = \gamma_{\x,v}(t)$ (with Jacobian $dM_\y = b\ dt\ dS(v)$), the $C$ term becomes an operator with bounded kernel, i.e. $L^2\to L^2$ bounded, with operator norm controlled by $\|q_2\|_\infty$, i.e. $\|Vh\|_{\infty}$. On to the $B$ term, we may write $|q_2(\x,v,t)-q_2(\x,v,0)| \le C t \|Vh\|_{\C^1}$ uniformly on $\D$, so that the kernel of $B$ has an integrable singularity and the $B$ term becomes an operator with integrable kernel, i.e. $L^2\to L^2$ bounded, with operator norm controlled by $\|Vh\|_{\C^1}$. On to the $A$ term, we assume without loss of generality to be working in isothermal coordinates. We change variable $(v,t)\mapsto \y(v,t) = \gamma_{\x,v}(t)$ to make appear
\begin{align*}
    Af(\x) = \int_{M} \frac{q_2(\x,v(\y),0)}{(d_g(\x,\y))^2} f(\y)\ dM_\y = \int_{M} \frac{q_2(\x,v(\y),0)}{(d_g(\x,\y))^2} e^{2\lambda(\y)} f(\y) \ d\y,
\end{align*}
where $d\y$ now represents the Lebesgue measure on $\Rm^2$. Expansions near $\x$ give that 
\begin{align*}
    d_g(\x,\y) &= e^{\lambda(\y)} |\x-\y| + |\x-\y|^2 d_1(\x,\y), \\
    \hat\theta(\y) &= \frac{\x-\y}{|\x-\y|} + |\x-\y| \hat\theta_1(\x,\y), 
\end{align*} 
this allows to rewrite $A$ as a Calder\'on-Zygmund operator of the form 
\begin{align}
    Af(\x) = \int_M \frac{q_2(\x,\hat\theta,0)}{|\x-\y|^2} f(\y)\ d\y + \int_M \frac{q_3(\x,\y)}{|\x-\y|}f(\y)\ d\y, \qquad \hat\theta := \frac{\x-\y}{|\x-\y|},
    \label{eq:Aterm}
\end{align}
where $q_3$ is uniformly bounded by $C \|Vh\|_{\C^1}$. By virtue of \cite[Theorem XI.3.1]{Mikhlin1980}, the first term together with the zero mean value condition $\int_{\Sm^1} q_2(\x,\hat\theta,0)\ d\theta = 0$ is an operator $L^2\to L^2$ continuous, with an operator norm bounded by $C\sup_{\x\in M} \|q_2(\x,\cdot)\|_{L^2(\Sm^1)}$, in turn bounded by $C\|V h\|_{\infty}$. The second term of \eqref{eq:Aterm} is another weakly singular operator whose operator norm can be bounded by $C\|Vh\|_{\C^1}$ as well.
\end{proof}

\bibliographystyle{siam}
\bibliography{../bibliography/bibliography}

\end{document}